\documentclass[3p]{elsarticle}

\usepackage{hyperref}
\usepackage[utf8]{inputenc}
\usepackage{amsmath}
\usepackage{amsthm}
\usepackage{amssymb}
\usepackage{mathrsfs}
\usepackage{amsfonts}%can be included
\usepackage{bbold}
\usepackage{enumerate}
\usepackage{graphicx}
\usepackage[normalsize]{subfigure}
\usepackage{xcolor}%[dvips]
\usepackage{url}
\usepackage{xifthen}

\RequirePackage[normalem]{ulem}
\journal{Journal of Computational and Applied Mathematics}

\newcommand{\phase}{100}

\newcommand{\spr}{\,\bigl|\bigr.\,}

\newcommand{\sR}{\set{R}}
\newcommand{\sN}{\set{N}}
\newcommand{\sZ}{\set{Z}}

\DeclareMathOperator{\nnz}{nnz}
\newcommand{\pyramid}{\ensuremath{\triangle}}
\newcommand{\squar}{\ensuremath{\square}}

\newcommand{\mesh}{\mathcal{M}_\VN}
\newcommand{\FEg}{\mathcal{X}_{\VN}}
\newcommand{\FE}{\mathcal{X}_{\VN,p}}
\newcommand{\FEo}{\mathcal{X}_{\VN,p,0}}

\newcommand{\skipifemptyarg}[1]{\ifthenelse{\isempty{#1}}{}{\left[#1\right]}}
\newcommand{\skipifscalar}[1]{\ifthenelse{\isempty{#1}}{}{;#1}}

\newcommand{\set}[1]{\mathbb{#1}}
\newcommand{\B}[1]{\boldsymbol{#1}} %bold
\newcommand{\V}[1]{\B{#1}} % vectors-matrices
\newcommand{\mat}[1]{\mathsf{#1}} %matrix after discretisation
\newcommand{\M}[1]{\mat{#1}}
\newcommand{\MB}[1]{\B{\M{#1}}}

\newcommand{\ol}[1]{\overline{#1}}

\newcommand{\iter}[1]{_{(#1)}}

\newcommand{\eff}{\mathrm{H}}

\newcommand{\x}{{\V{x}}}

\newcommand{\ZNd}{\set{Z}^d_{\VN}}
\newcommand{\ZtNd}{\set{Z}^d_{\tVN}}
\newcommand{\cE}{\mathscr{E}}
\newcommand{\cEN}{\cE_\VN}
\newcommand{\xEN}{\set{E}_\VN}
\newcommand{\xhVN}{{\set{C}^{d\times\VN}}}
\newcommand{\xhMN}{[\xhVN]^2}
\newcommand{\xVtN}{{\set{R}^{d\times\tVN}}}
\newcommand{\xMtN}{[\xVtN]^2}

\newcommand{\cT}{\mathscr{T}}
\newcommand{\cP}{\mathcal{P}}
\newcommand{\cTNd}{\cT_\VN^d}

\newcommand{\Hnp}{H^1_{\zmean}}
\newcommand{\zmean}{0}

\providecommand{\norm}[1]{\lVert#1\rVert}
\newcommand{\scal}[2]{\bigl(#1,#2\bigr)}
\newcommand{\mean}[1]{\langle #1 \rangle}
\newcommand{\bilf}[2]{a\ifthenelse{\isempty{#1}}{}{\bigl(#1,#2\bigr)}}

\newcommand{\Aeff}{A_{\eff}}
\newcommand{\tA}{\widetilde{A}}
\newcommand{\tAeff}{\tA_{\eff}}
\newcommand{\Aeffem}{A_{\eff,\VN,p}}
\newcommand{\Aeffth}{A_{\eff,\VN}}
\newcommand{\AeffN}{A_{\eff,\VN}}

\newcommand{\lam}{\ensuremath{\lambda}}
\newcommand{\del}{\ensuremath{\delta}}
\newcommand{\alp}{\ensuremath{\alpha}}

\newcommand{\mI}{\ensuremath{\B{I}}}
\newcommand{\Vk}{\ensuremath{{\V{k}}}}
\newcommand{\Vm}{\ensuremath{{\V{m}}}}
\newcommand{\Vl}{\ensuremath{{\V{l}}}}
\newcommand{\Vn}{\ensuremath{{\V{n}}}}
\newcommand{\MBA}{\ensuremath{\B{\M{A}}}}
\newcommand{\MBG}{\ensuremath{\B{\M{G}}}}
\newcommand{\MBF}{\ensuremath{\B{\M{F}}}}
\newcommand{\MBFi}{\ensuremath{\B{\M{F}}^{-1}}}
\newcommand{\MBe}{\ensuremath{\B{\M{e}}}}
\newcommand{\MBE}{\ensuremath{\B{\M{E}}}}
\newcommand{\MBv}{\ensuremath{\B{\M{v}}}}
\newcommand{\MBx}{\ensuremath{\B{\M{x}}}}

\newcommand{\mB}{\ensuremath{\B{B}}}

\newcommand{\Vo}{\ensuremath{{\V{0}}}}
\newcommand{\MBhG}{\ensuremath{{\B{\M{\hat{G}}}}}}
\newcommand{\VN}{\ensuremath{{\V{N}}}}
\newcommand{\tVN}{\ensuremath{{2\V{N}-1}}}
\newcommand{\pVN}{|\VN|_{\Pi}}

\newcommand{\Ve}{\ensuremath{\V{e}}}
\newcommand{\VE}{\ensuremath{\V{E}}}
\newcommand{\Vu}{\ensuremath{\V{u}}}
\newcommand{\Vv}{\ensuremath{\V{v}}}
\newcommand{\Vw}{\ensuremath{\V{w}}}

\newcommand{\puc}{\mathcal{Y}}

\newtheorem{theorem}{Theorem}[section]
\newtheorem{lemma}[theorem]{Lemma}

\newtheorem{corollary}[theorem]{Corollary}

\DeclareMathOperator{\argmin}{argmin}

\DeclareMathOperator{\D}{d\!}
\DeclareMathOperator{\curl}{curl}

\newcommand{\imu}{\mathrm{i}}		%  Imaginary Unit

\DeclareMathOperator{\rank}{rank}%kernel of operators
\begin{document}

\begin{frontmatter}

\title{Energy-based comparison between the Fourier--Galerkin method and the finite element method}
%\tnoteref{mytitlenote}

\author[address_jv]{Jaroslav Vond\v{r}ejc}
\address[address_jv]{Technische Universität Braunschweig, Institute of Scientific Computing, Mühlenpfordtstra{\ss}e~23, D-38106 Braunschweig, Germany}
\ead{j.vondrejc@tu-bs.de; vondrejc@gmail.com}
%\ead{vondrejc@gmail.com}


\author[address_TdG]{Tom~W.J.~de Geus}
\address[address_TdG]{\'{E}cole Polytechnique F\'{e}d\'{e}rale de Lausanne (EPFL), Institute of Theoretical Physics, CH-1015 Lausanne, Switzerland}
\ead{tom@geus.me}



\begin{abstract}
The Fourier-Galerkin method (in short FFTH) has gained popularity in numerical homogenisation because it can treat problems with a huge number of degrees of freedom.
Because the method incorporates the fast Fourier transform (FFT) in the linear solver, it is believed to provide an improvement in computational and memory requirements compared to the conventional finite element method (FEM).
Here, we systematically compare these two methods using the energetic norm of local fields, which has the clear physical interpretation as being the error in the homogenised properties.
This enables the comparison of memory and computational requirements at the same level of approximation accuracy.
We show that the methods' effectiveness relies on the smoothness (regularity) of the solution and thus on the material coefficients.
Thanks to its approximation properties, FEM outperforms FFTH for problems with jumps in material coefficients, while ambivalent results are observed for the case that the material coefficients vary continuously in space.
FFTH profits from a good conditioning of the linear system, independent of the number of degrees of freedom, but generally needs more degrees of freedom to reach the same approximation accuracy.
More studies are needed for other FFT-based schemes, non-linear problems, and dual problems (which require special treatment in FEM but not in FFTH).
\end{abstract}


\end{frontmatter}

%\linenumbers

%=====================
\section{Introduction}
%=====================
A physical system (e.g.\ thermal or mechanical) with a complex microstructure is typically represented with (ensembles of) periodic microstructural cells 
(see examples by us \cite{DeGeus2016stat,DeGeus2015a,DeGeus2019fric,NeKrVo2013twoscale,NeKrVo2013CCC}).
Its mathematical modelling and concurrent numerical simulation, which usually involves solving one or more (coupled) partial differential equations, may be computationally very demanding.
This is especially true for problems including uncertainties, geometric and/or material nonlinearities, or multiple scales, because the microstructural problems have to be solved many times.
Here, we compare the numerical efficiency of the now very popular Fourier-Galerkin methods to the classical finite element method (FEM).
In particular, we test the popular believe that the former offer a significant higher efficiency than the latter, for a specific set of problems. 
Specifically, FEM uses a discretisation that is piece-wise continuous in space, allowing arbitrary discretisation morphologies as well as boundary conditions, while the Fourier-Galerkin method is based on a discretisation using globally supported trigonometric polynomials \cite{Boyd2000book,Canuto2006SM_book,SaVa2000PIaPDE,Cain1984,Cai1989,Luciano1994,VoZeMa2014FFTH}, allowing the use of the fast Fourier transform (FFT) typically for regular and periodic discretisations.
Our contribution focuses on the clarification of the numerical efficiency of Fourier-Galerkin methods with respect to its approximation accuracy.
We compare the two methods through the accuracy of local fields, which is based on errors of local fields measured using the energetic norm that also correspond to the errors in effective properties\footnote{Note that this is only valid for a conforming approximation, i.e.\ when compatible fields are used and the integrals in the Galerkin formulation are evaluated exactly, see \ref{sec:energetic-norm} for details.}. This gives the comparison a clear physical meaning.

The Fourier-Galerkin method is part of a class of methods that employs FFT in the context of homogenisation \cite{Boyd2000book,Canuto2006SM_book,SaVa2000PIaPDE,Moulinec1994FFT,VoZeMa2014FFTH}, as was first introduced by Moulinec and Suquet \cite{Moulinec1994FFT} who present an efficient algorithm that relies on a fixed-point iterative solution of the Lippmann-Schwinger (L-S) equation, thereby employing a Green operator. This approach was recognised as a collocation method \cite{ZeVoNoMa2010AFFTH} and later, thanks to a \emph{variational reformulation} of the L-S equation \cite{VoZeMa2012LNSC,VoZeMa2014FFTH}, as a \emph{Galerkin approximation with numerical integration} (denoted FFTH-GaNi) \cite{VoZeMa2014FFTH,VoZeMa2014GarBounds}. Note that this concerns merely a variational justification of Ref.~\cite{Moulinec1994FFT} and that both provide exactly the same discrete solution. Moreover, the variational reformulation has allowed a further development using standard mathematical tools such as convergence studies \cite{VoZeMa2014FFTH,Schneider2014convergence}, the usage of more efficient linear solvers \cite{MiVoZe2016jcp} and a straightforward treatment of non-linear small \cite{ZeGeVoPeGe2017} and finite \cite{GeVoZePeGe2017large} strain mechanics; rendering it a competitive alternative to FEM.

Here we will compare the FEM to FFTH-GaNi and also to the Fourier-Galerkin method based on \emph{exact integration} (denoted FFTH-Ga) because of its improved accuracy compared to the scheme with numerical integration.
The exact evaluation of the corresponding integrals (and also the energetic norms) in closed form \cite{VoZeMa2014GarBounds,Vondrejc2015FFTimproved} allows us to eliminate the numerical error of the quadrature, and thus to purely focus on discretisation and algebraic errors\footnote{By algebraic errors we  consider the error between the solution of the linear system and its approximation using e.g.\ conjugate gradients.} in the comparison.
Even though several papers \cite{Michel1999,Prakash2009,Liu2010comparison,Dunant2013a,Robert2015,Leclerc2015,Willot2015,Schneider2015stag,Anglin2014} compare FFT-based methods to FEM and other discretisation methods, the effectiveness of predicting local fields at the same level of accuracy was never assessed.
In particular, the accuracy of local fields was but cannot be deduced from the homogenised properties for nonconforming methods \cite{Michel1999,Prakash2009,Liu2010comparison,Dunant2013a,Robert2015,Leclerc2015}.
Consequently, a very good prediction of energies (homogenised properties) may be obtained with local fields that still possess a big error, because those approaches (e.g.~\cite{Moulinec1994FFT}) violate the variational structure of the problem, see \cite{Strang1972varcrime,VoZeMa2014GarBounds,Vondrejc2015FFTimproved}.
Those limitations were overcome by comparing various FFT-based schemes against analytical solutions \cite{Willot2015,Schneider2015stag,Anglin2014}, however, without focussing on the computational efficiency as is our focus.

We limit our comparison to FEM and the Fourier-Galerkin method as both numerical methods can be used in broad class of microscopic studies (e.g.\ \cite{DeGeus2016stat,DeGeus2015a,DeGeus2019fric,NeKrVo2013twoscale,NeKrVo2013CCC}).
These methods can additionally be used in multi-scale analyses, where the effectiveness of both methods in comparison to other approaches (e.g.~\cite{Hou2009,Kanoute2009,Hughes1996,Efendiev2013}) is left undiscussed.
Similarly we leave the effectiveness of other improvements and alternative approaches to FFT-based methods (such as the variational discretisation of the Lippmann-Schwinger equation with piece-wise constant basis functions \cite{Brisard2010FFT,Brisard2012FFT,Brisard2016}, the modification of the integral kernel \cite{Willot2013fourier,Willot2015}, schemes derived from finite differences \cite{Schneider2015stag} and hexahedral finite elements \cite{Schneider2016hexa}, or smoothing of the material coefficients \cite{Vondrejc2013PhD,Kabel2015}) undiscussed.
Although those methods are promising, they do not allow for the direct evaluation of guaranteed bounds on homogenised properties, which we have argued to be the crux of our comparison.

The remainder of this paper is structured as follows.
The problem is introduced in Section~\ref{sec:problem-description}.
Our results are presented in Section~\ref{sec:results}, followed by our conclusions and a short discussion in Section~\ref{sec:conclusion}. All mathematical details, including the numerical treatment, can be found in \ref{sec:continuous-formulation} through \ref{sec:approximate-homogenised-properties}. In terms of nomenclature, we denote $d$-dimensional vectors and matrices by boldface letters: $\V{a} = \left(a_i \right)_{i=1,2,\ldots,d}\in \sR^d$ or $\B{A} = (A_{ij})_{i,j=1,2,\ldots,d} \in \sR^{d\times d}$. Matrix-matrix and matrix-vector multiplications are denoted as $\B{C} = \B{A}\B{B}$ and $\V{c} = \B{A}\V{b}$, which in Einstein's summation notation reads $C_{ik} = A_{ij}B_{jk}$ and $b_{i} = A_{ij} b_j$ respectively.
The Euclidean inner product will be referred to as $c = \V{a} \cdot \V{b} = a_i b_i$, and the induced norm as $\| \V{a} \| = \sqrt{\V{a}\cdot\V{a}}$. Vectors and matrices such as $\MB{x}$, $\MB{b}$, and $\MBA$ arising from discretisation are denoted by the bold serif font in order to highlight their special structure. The space of $\puc$-periodic continuous functions defined on a periodic cell $\puc=(-\frac{1}{2},\frac{1}{2})^d$ is denoted as $C(\puc)$ and the space of square integrable $\puc$-periodic functions as $L^2(\puc)$. The analogical space $L^2(\puc;\sR^d)$ collects $\sR^d$-valued functions $\Vv:\puc\rightarrow\sR^d$ with components $v_i$ from $L^2(\puc)$. Finally, $\Hnp(\puc)=\{v\in L^2(\puc) \spr \nabla v\in L^2(\puc;\sR^d), \int_{\puc} v(\x) \D{\x} = 0\}$ denotes the Sobolev space of periodic functions with zero mean.

%=====================
\section{Problem description}
\label{sec:problem-description}
%=====================

%=====================
\subsection{Model problem}
\label{sec:problem-description-model}
%=====================

A model problem in homogenisation \cite{Bensoussan1978book} consists of a scalar linear elliptic variational problem defined on a unit domain $\puc=(-\frac{1}{2},\frac{1}{2})^d$ in $d$-dimensional setting, for which we consider both $d=2$ and $d=3$.
Before proceeding, we introduce a compact notation of a bilinear form $a:L^2(\puc;\sR^d)\times L^2(\puc;\sR^d)\rightarrow\sR$ defined as
\begin{align*}
\bilf{\Vv}{\Vw} :=
\int_\puc \B{A}(\x) \Vv(\x) \cdot \Vw(\x) \D{\x} = \int_\puc A_{ij}(\x) v_j(\x) w_i(\x) \D{\x},
\end{align*}
where $\B{A}:\puc\rightarrow\sR^{d\times d}$ is a symmetric and uniformly elliptic second order tensor of material coefficients.

We now consider the homogenisation problem to find effective material properties $\V{A}_\eff\in\sR^{d\times d}$. It is based on the minimisation of a microscopic energetic functional for constant vectors $\VE\in\sR^d$ as
\begin{equation}
\label{eq:homog}
\V{A}_\eff \VE \cdot \VE
=
\min_{v\in\Hnp(\puc)} \bilf{\VE+\nabla v}{\VE+\nabla v}
\end{equation}
where $\nabla v$ is the periodically fluctuating microscopic field of $\VE$.
In the sequel we consider exclusively $\VE = (\delta_{1,i})_{i=1}^d\in\sR^d$ (i.e.\ in 3D $\VE = [1,0,0]$); therefore, the $11$-component of the homogenised properties
\begin{align*}
\V{A}_\eff \VE \cdot \VE = \V{A}_{\eff,11} =: \Aeff
\end{align*}
will be of particular interest.

Both discretisation methods, FEM and FFTH (described in more detail in \ref{sec:finite-element-method} and \ref{sec:fourier-galerkin-method}), are conforming methods approximating the space $\Hnp(\puc)$ with some finite-dimensional subspace $\FEg\subset\Hnp(\puc)$, where $\VN$ is a discretisation parameter.
In the case of FFTH, $\FEg$ is spanned by the basis $\{\psi_i\}_{i=1}^n$ that consists of trigonometric polynomials, while in the case of FEM the basis consists of continuous piece-wise polynomials.
The discretised problem is then defined using the Ritz-Galerkin method
\begin{align}
\label{eq:homog_approx}
\AeffN
&= \min_{v_\VN\in X_\VN}\bilf{\VE+\nabla v_\VN}{\VE+\nabla v_\VN}=\bilf{\VE+\nabla u_\VN}{\VE+\nabla u_\VN}.
\end{align}
The minimiser
$u_\VN=\sum_{i=1}^n \M{u}_{i} \psi_i$ is described with the optimality condition (also known as the weak formulation)
\begin{align*}
\bilf{\nabla u_\VN}{\nabla v_\VN} = -\bilf{\VE}{\nabla v_\VN}
\quad\forall\, v_\VN \in \FEg.
\end{align*}
The coefficients $\MB{u} = (\M{u}_{i})_{i=1}^n$ with respect to the basis $\{\psi_i\}_{i=1}^n$ are obtained from the linear system
\begin{equation}
\label{eq:linsys_FEM_body}
\MBA\MB{u} = \MB{b},
\quad
\MBA_{ij}  = \bilf{\nabla\psi_j}{\nabla\psi_i},
\quad
\MB{b}_{i} = -\bilf{\VE}{\nabla\psi_i}
\quad
\text{for }i,j=1,2,\dotsc,n.
\end{equation}

Because of computational effectiveness, FFTH does not solve the problem for the state $u_\VN:\puc\rightarrow\sR$ but for its gradient $\Ve_\VN = \nabla u_\VN:\puc\rightarrow\sR^d$.
Since the gradient fields do not span the whole space of vector-valued functions, the compatibility condition has to be additionally enforced.
It is provided here with the discrete projection operator $\MBG$ (evaluated using FFT) along with some suitable linear solver.
For the latter we employ the conjugate gradient (CG) method, but for the purpose of ensuring compatibility other methods such as Richardson iteration or Chebyshev's method would be suitable too.

For FFTH, we will consider two discretisation schemes.
The first scheme, denoted as FFTH-Ga (see \ref{sec:FFTHga} for details), is based on exact integration of \eqref{eq:homog_approx} and leads to the following linear system
\begin{equation*}
%\label{eq:linsys_FFTH_body}
\MBG \MBA \MBe = \MB{b} \qquad\text{with } \MB{b} =-\MBG \MBA \MBE.
\end{equation*}
The second scheme, denoted as FFTH-GaNi (see \ref{sec:FFTHgani} for details), is fully equivalent to the original Moulinec-Suquet scheme \cite{Moulinec1994FFT} (in fact, the solution vectors coincide \cite{VoZeMa2012LNSC}).
It is based on numerical integration of Eq.\ \eqref{eq:homog_approx} and leads to a similar linear system, however, with a different matrix $\widetilde{\MB{A}}$.
Both $\MBA$ and $\widetilde{\MBA}$ are block diagonal matrices.
Because only this `diagonal' is stored, no special sparse matrix storage is needed.
The difference between the two is that $\MBA$ is expressed on a double discretisation grid while $\widetilde{\MBA}$ necessitates only a single grid.
Consequently, the higher accuracy of FFTH-Ga comes at the costs of higher computational and memory requirements compared to FFTH-GaNi.

% ==============================================================================
\subsection{Discretisation error}
\label{sec:discretisation-error}
% ==============================================================================

For the problem considered here, the \emph{discretisation error} between the exact $\nabla u=\Ve$ and the approximate solution $\nabla u_\VN = \Ve_\VN$, respectively obtained from Eqs.~\eqref{eq:homog} and \eqref{eq:homog_approx}, is measured
in terms of the energetic norm, defined as
\begin{equation*}
\|\Ve\|_{\B{A}} = \sqrt{\bilf{\Ve}{\Ve}}.
\end{equation*}
It has a clear physical meaning in the case of conforming methods: the square of discretisation error is equal to the error in the homogenised properties
\begin{equation}
\label{eq:discretisation_error}
\|\Ve-\Ve_\VN\|_{\B{A}}^2 = \AeffN - \Aeff,
\end{equation}
see \ref{sec:energetic-norm} for a derivation.
In a similar way, it is possible to compute an algebraic error as an error between the exact solution of a linear system and approximate solution of the linear system, i.e. algebraic energetic error during the iterations of the conjugate gradients.

This approach is motivated by the guaranteed bounds approach presented e.g.\ in \cite{Willis1989,Dvorak1993master,Haslinger1995optimum} and later elaborated for FFTH in \cite{VoZeMa2014GarBounds,Vondrejc2015FFTimproved}.
We note that FFTH-GaNi does not provide the bound directly but can be evaluated in the post-processing stage.
For the numerical evaluation of Eq.~\eqref{eq:discretisation_error}, the homogenised properties $\Aeff$ have been estimated according to \ref{sec:approximate-homogenised-properties};
we emphasise that this estimation has no impact on the comparison because the value $\Aeff$ or its approximation is only a constant which influences only the absolute value of the error.
Note that the values are chosen sufficiently accurate to catch the trends in the figures.
We also note that the inequality $\Aeff\leq\AeffN$ between homogenised properties and its approximation is valid for all conforming methods based on the Galerkin approximation with exact integration; a numerical integration can cause violation of the inequality due to the so-called variational crime \cite{VoZeMa2014GarBounds,Vondrejc2015FFTimproved,Haslinger1995optimum}.

Generally, the convergence of the approximate solutions
\begin{align*}
\|\Ve-\Ve_\VN\|_A \rightarrow 0
\quad\text{for }\min_{\alp}N_\alp \rightarrow\infty
\end{align*}
can be arbitrarily slow when we increase the discretisation grid \cite{Ciarlet1991HNAFEM}.
When the solution has a higher regularity, meaning that a solution $\nabla \Vu=\Ve$  is in e.g. $\Ve\in H^s(\puc)$ for some positive $s$, the Fourier-Galerkin solutions converge with a rate $s$
\begin{align*}
\sqrt{\Aeffth-\Aeff}=\|\Ve-\Ve_\VN\|_A \leq  C N^{-s} \|\Ve\|_{H^s(\puc)},
\end{align*}
for details see e.g.\ \cite{VoZeMa2014FFTH,Schneider2014convergence,Vondrejc2013PhD}.

For FEM, the convergence rate, furthermore, depends on the order, $p$, of the polynomial approximation
\begin{align*}
\sqrt{\Aeffem-\Aeff} = \|\nabla \Vu-\nabla\Vu_{\VN,p}\|_A \leq  \tilde{C} N^{-r} \|\Ve\|_{H^s(\puc)}
\quad\text{for }
r=\min\{s,p\},
\end{align*}
which requires to use higher order polynomials to obtain higher convergence rate \cite{Ciarlet1991HNAFEM}. Note that the constants $C$ and $\tilde{C}$ are independent of the discretisation parameter $N$.

% ====================================================================
\subsection{Linear solvers}
\label{sec:linear-solvers}

The linear system that arises from FEM discretisation (Eq.~\eqref{eq:linsys_FEM_body}) can be solved using a direct or an iterative solver.
The direct solver is based either on LU decomposition or on Cholesky decomposition, which is a variant suitable for symmetric matrices.
In the latter, one decomposes the matrix $\MBA$ to
\begin{equation*}
\MBA = \MB{L} \MB{L}^T
\end{equation*}
whereby $\MB{L}$ is a sparse lower triangular matrix. Using this decomposition the solution of Eq.\ \eqref{eq:linsys_FEM_body} can be trivially found in two steps.
The actual decomposition is performed by the fill-reducing variant of the Cholesky decomposition\footnote{For instance available at: \href{https://github.com/scikit-sparse}{https://github.com/scikit-sparse}.}. It reduces the memory requirements of the factors, which are, like $\MBA$, stored in the CSR (compressed sparse row) format.

The advantage of these direct solvers is that the computationally consuming decomposition is independent of the right-hand side.
In the case of linear problems, like in this paper, it therefore has to be done only once.
The drawback of the direct solvers lies in significantly higher memory requirements for the factors than for the original system.
This can be partly overcome by the so-called pivoting
\begin{align*}
%\label{eq:cholesky}
\MBA = \MB{P}\MB{L}\MB{L}^T\MB{P}^T
\end{align*}
in which the sparsity of the decomposition is increased by using a permutation matrix $\MB{P}$; this approach is incorporated in this paper.
It is remarked that a direct solver is unsuitable for FFTH because the system is never assembled.

The memory drawbacks of direct solvers are overcome by iterative solvers, which are suitable for both FFTH and FEM. Here we use the conjugate gradient (CG) method because of its effectiveness to minimise the energetic error.
Other methods such as Chebyshev's method could also be suitable especially for parallel computing, see \cite{MiVoZe2016jcp} for a comparison of linear solvers for FFTH.
Based on the linear system in Eq.~\eqref{eq:linsys_FEM_body}, the idea behind the CG method is that the solution $\MB{u}$ is approximated as a linear combination
\begin{equation*}
%\label{eq:sol:cg}
\MB{u} \approx \MB{u}\iter{k} = \sum\limits_{i=1}^k \alpha_{(i)} \MB{p}_{(i)}
\end{equation*}
with $\MBA$-orthogonal basis vectors $\MB{p}\iter{i}$.
The crux of the method is that the factors $\alpha_{(k)}$ and the bases $\MB{p}_{(k)}$ are iteratively constructed using only the last residual and the last basis vector $\MB{p}\iter{k-1}$ along with the (costly) matrix-vector multiplication.
Therefore, only the original system and additionally three vectors have to be stored, leading to memory savings compared to a direct solver.

Fortunately, the approximate solution $\MB{u}_{(k)}$ typically reaches a sufficient accuracy for $k \ll n$, whereby the number of iterations, $k$, depends on the condition number $\kappa$ of $\MBA$ on a subspace $\xEN$.
Particularly, the number of CG iterations, with respect to some accuracy, is bounded with the following (non-optimal) estimate
\begin{align*}
\|\underbrace{\MB{u}-\MB{u}\iter{k}}_{\substack{\text{error in $k$-th}\\\text{iteration}}}\|_{\MBA}^2 \leq 4 \left(\frac{\sqrt{\kappa}-1}{\sqrt{\kappa}+1}\right)^{2k} \|\underbrace{\MB{u}-\MB{u}\iter{0}}_{\text{initial error}}\|_{\MBA}^2,
\end{align*}
with the discrete energetic norm $\|\MBx\|_{\MBA}=\sqrt{\MBA\MBx\cdot\MBx}$.
We emphasise that the condition number is known a priori for FFTH, while it is not easily evaluated for FEM.

As will be demonstrated below, for FFTH $\kappa$ is usually low (i.e.\ the system is well conditioned) and independent on the number of degrees-of-freedom.
For FEM, however, $\kappa$ is often quite high and increases with the number of degrees-of-freedom.
This frequently leads to a high number of CG iterations, and therefore to numerical inaccuracies due to round-off errors.
The poor conditioning of the linear system can be overcome by a suitable preconditioner $\MB{M}$, which is an easily invertible matrix approximating the original matrix $\MB{A}$.
Instead of the original system, we solve
\begin{align*}
\MB{M}^{-1}\MB{A}\MB{u} = \MB{M}^{-1}\MB{b}.
\end{align*}
For the preconditioner, we incorporate incomplete Cholesky decomposition
\begin{align*}
\MB{M}=\widetilde{\MB{L}}\widetilde{\MB{L}}^T \approx \MBA,
\end{align*}
particularly we consider the zero-fill variant that builds on the same sparsity pattern of the original system $\MBA$ as for the incomplete Cholesky factor $\widetilde{\MB{L}}$.

%=====================
\section{Results}
\label{sec:results}
%=====================

The performance of both methods depends on the regularity of the solution, and thus on the regularity of material coefficients.
Therefore, two academic examples that are naturally defined on a regular grid are considered in Sections~\ref{sec:sizes-of-linear-systems}-\ref{sec:computational-effectiveness-of-the-conjugate-gradient-solver}: (i) one with a jump in coefficients and (ii) one with continuous coefficients.
Later, in Section~\ref{sec:examples:extra} we consider two additional examples:
(iii) A circular inclusion is used to investigate a smooth inclusion boundary, which is favourable for the flexibility of conforming discretisations for FEM.
(iv) An image-based geometry is finally used as it is generally believed to be favourable for FFTH because of its natural definition on regular grids.
All examples are constructed in a way that allows the evaluation of guaranteed bounds on homogenised properties, serving as an energetic criterion.

From the implementation perspective, we rely on open-source tools for both methods: the FEniCS project\footnote{Available from: \href{www.fenicsproject.org}{fenicsproject.org}} \cite{Fenics_v15} for FEM and FFTHomPy\footnote{Available from: \href{https://github.com/vondrejc/FFTHomPy}{github.com/vondrejc/FFTHomPy}} for FFTH.
Moreover the simple code that calculates the homogenisation problem using FEM, FFTH-Ga, and FFTH-GaNi is attached as supplementary material. It enables the reproducibility of the results presented in the manuscript.

% ==============================================================================
\subsection{Material parameters}
\label{sec:material-parameters}

% ==============================================================================
\begin{figure}[htb]
 \centering
 \subfigure[Square inclusion \squar]{
  \includegraphics[width=0.46\textwidth]{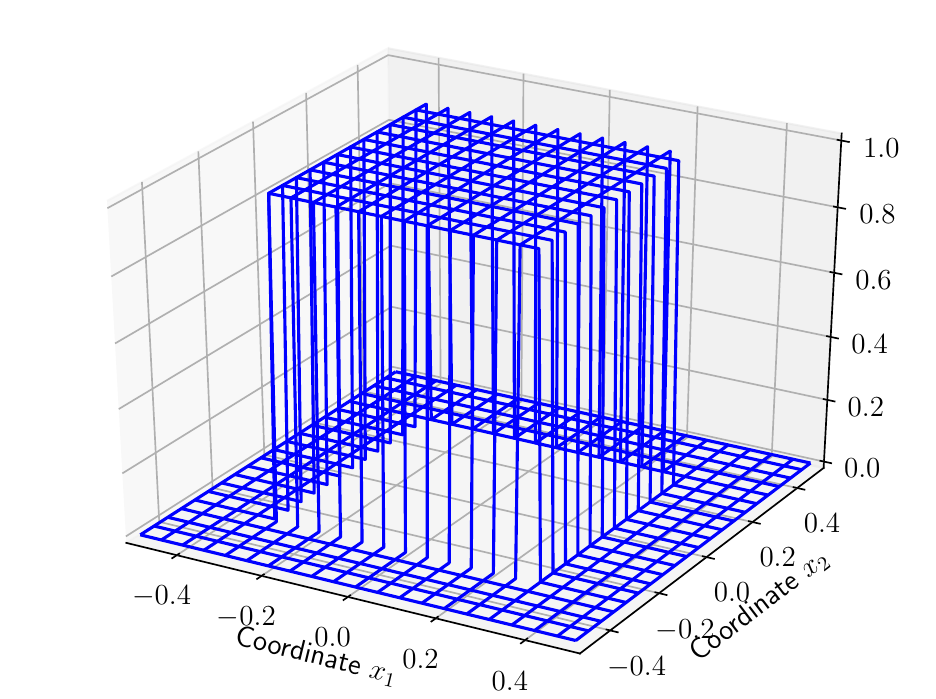}
 }
 \subfigure[Pyramid inclusion \pyramid]{
  \includegraphics[width=0.46\textwidth]{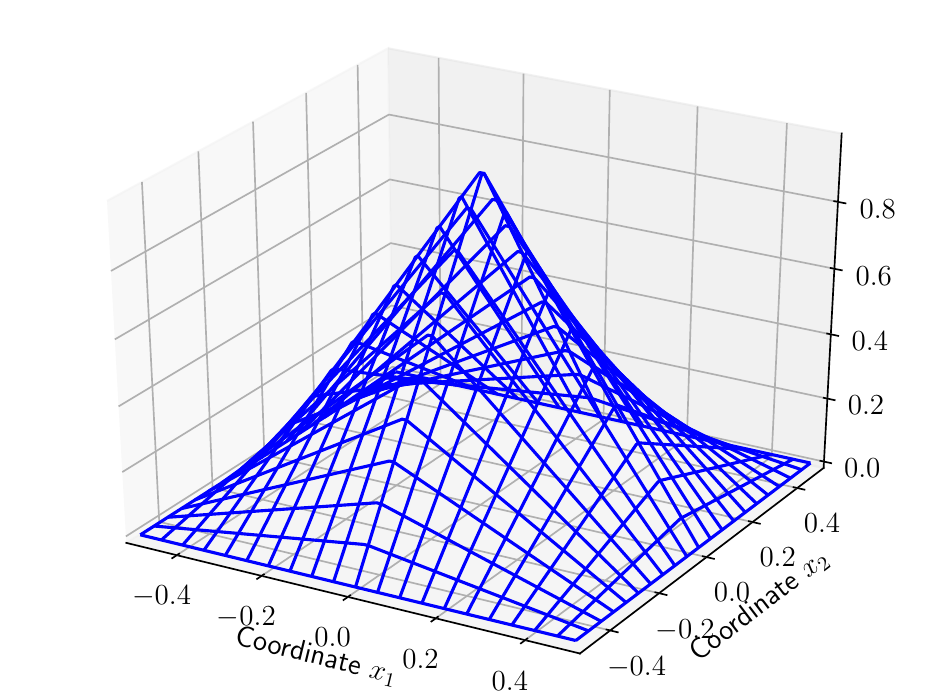}
 }
 \caption{Two-dimensional square- and pyramid-like inclusions defined by Eq.\ \eqref{eq:shape_fun}.}
 \label{fig:shape_fun}
\end{figure}

Here, we define
\begin{align*}
\B{A}(\x) &= \B{M}_{(d)} + \rho \mI f_\bullet(\x)
\end{align*}
where $\mI\in\sR^{d\times d}$ is the identity matrix, the parameter $\rho\in\{10,100\}$ corresponds to a material contrast, and  $\B{M}_{(d)}\in\sR^{d\times d}$ introduces some anisotropy. Its coefficients in 2D and 3D are set to
\begin{align*}
\B{M}_{(2)} =
\left[\begin{matrix}\frac{7}{4} & \frac{\sqrt{3}}{4}\\\frac{\sqrt{3}}{4} & \frac{5}{4}\end{matrix}\right]
\quad\text{and}\quad
\B{M}_{(3)} =
\left[\begin{matrix}\frac{31}{16} & \frac{5 \sqrt{3}}{16} & \frac{3}{8}\\\frac{5 \sqrt{3}}{16} & \frac{21}{16} & \frac{\sqrt{3}}{8}\\\frac{3}{8} & \frac{\sqrt{3}}{8} & \frac{11}{4}\end{matrix}\right]
\end{align*}
which have eigenvalues $\{1,2\}$ and $\{1,2,3\}$, respectively. Finally, functions $f_{\bullet}:\puc\rightarrow\sR^d$ defined as
\begin{subequations}
\label{eq:shape_fun}
\begin{align}
f_{\square}(\x) &=
\begin{cases}
1&\text{for }\x \text{ such that }x_i<0.3\text{ for }i=1,\dotsc,d
\\
0&\text{otherwise}
\end{cases},
\\
f_{\triangle}(\x) &=
\prod_{i=1}^d (1-2x_i),
\end{align}
\end{subequations}
describe the shape of the inclusions; these square- and the pyramid-like geometries are depicted in 2D in Figure~\ref{fig:shape_fun}.
These functions are chosen because they allow exact numerical integration using both FEM and FFTH, see \ref{sec:finite-element-method} and \ref{sec:fourier-galerkin-method}; the influence of numerical integration is thus fully eliminated.

\begin{figure}[htb]
 \centering
 \includegraphics[height=0.34\textwidth]{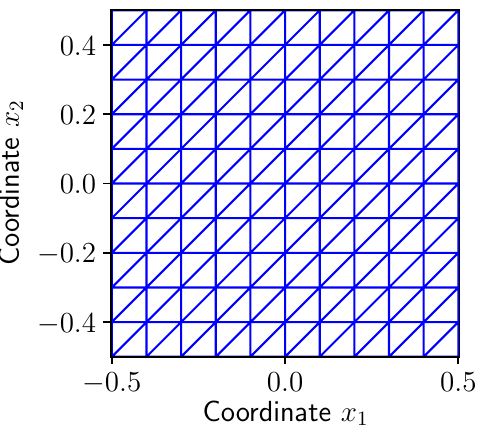}
 \caption{A regular triangular finite element mesh, with $\VN=[N,N]=[10,10]$ elements in the two-dimensional setting. (It concerns a generic mesh, the material contrast is not shown.)}
 \label{fig:mesh}
\end{figure}

For both methods, the discretisation is considered the same in all the spatial directions, e.g. in 3D $\VN=(N,N,N)$.
In the case of FFTH, exact quadrature is possible regardless the value of the discretisation parameter $\VN$, which corresponds to the number of grid points.
However, for FFTH-GaNi, favourable behaviour is obtained for odd discretisations that comply with the geometry, see \cite{VoZeMa2014GarBounds,Vondrejc2015FFTimproved};
therefore, we use $N = 5, 15, 45, 135, 405$.

In the case of FEM, a regular triangular mesh is considered, see Figure~\ref{fig:mesh} for a 2D example.
Here, the discretisation parameter $\VN$ corresponds to the number of elements --- inversely proportional to their characteristic size.
For exact quadrature, the mesh has to comply with the discontinuities in material coefficients or their derivatives; therefore multiples of 10 are used, e.g.\ $N=10, 20, 30$.
In the sequel we perform a detailed analysis is 2D, while, due to the computational demands, an analysis comprising fewer samples is performed in 3D.

\subsection{Sizes of linear systems}\label{sec:sizes-of-linear-systems}
Here, the discretisation error is expressed as a function of the size of linear systems.
For FEM, the system size is determined by the dimension of a discretisation space.
A more complicated situation arises for FFTH, for which the dimension of the approximation space is always smaller than the size of linear system.
It is caused by the necessity to enforce the compatibility condition, see \ref{sec:fourier-galerkin-method} for details.

\begin{figure}[htb]
	\centering
	\subfigure[\squar, $d=2$, $\rho=\phase$]{
		\includegraphics[width=0.46\textwidth]{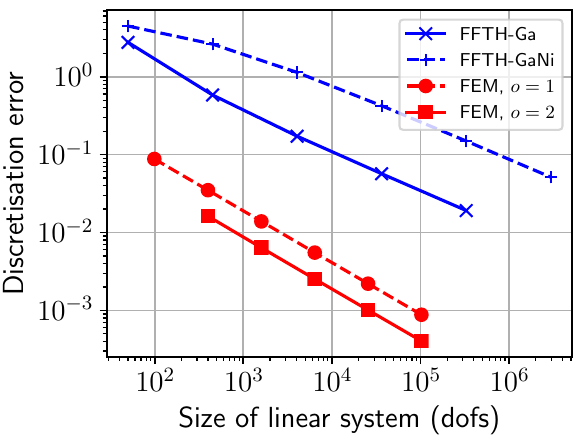}
	}
	\subfigure[\pyramid, $d=2$, $\rho=\phase$]{
		\includegraphics[width=0.46\textwidth]{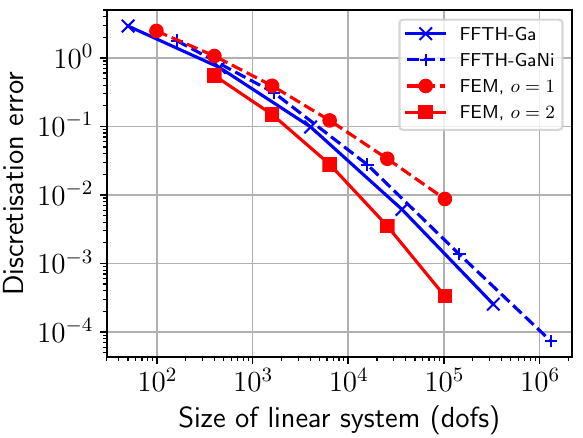}
	}
	\subfigure[\squar, $d=3$, $\rho=\phase$]{
		\includegraphics[width=0.46\textwidth]{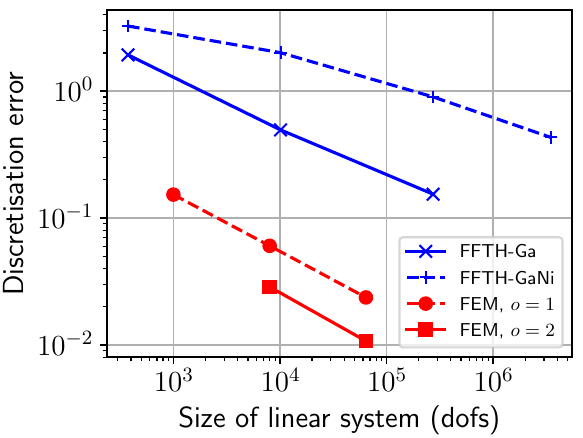}
	}
	\subfigure[\pyramid, $d=3$, $\rho=\phase$]{
		\includegraphics[width=0.46\textwidth]{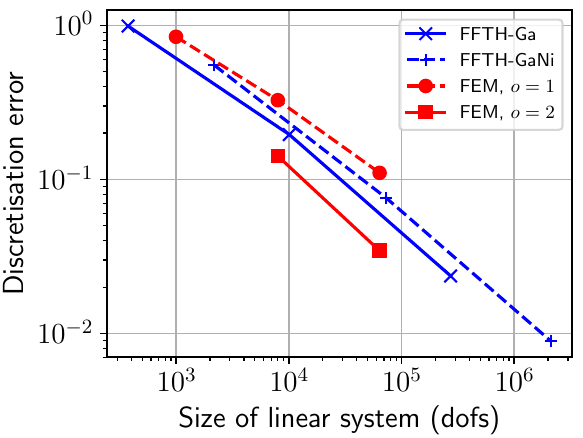}
	}
	\caption{Dependence of the approximate discretisation error in the energetic norm on the sizes of linear systems (see Eq.~\eqref{eq:discretisation_error}).
		The left figures are obtained using the square inclusion, the right figures using the pyramid inclusion. The upper figures are in $d = 2$ while the bottom ones are in $d = 3$.}
	\label{fig:linsys_size}
\end{figure}

The numerical results, compared in Figure~\ref{fig:linsys_size}, provide similar results for the 2D and the 3D problems, cf.\ Figures~\ref{fig:linsys_size}(a,b) to \ref{fig:linsys_size}(c,d).
As already reported in \cite{Vondrejc2015FFTimproved}, FFTH-Ga is always better than FFTH-GaNi, which is caused by the inconsistency error of numerical integration.
As observed for all cases in FEM, the second order polynomials have better approximation properties than the first order polynomials.
FFTH has better approximation properties when the solution is more regular, such as in the case of the pyramid inclusion.
For the square inclusion, in Figure~\ref{fig:linsys_size}(a,c), FEM requires a significantly smaller linear system than FFTH to attain the same discretisation error, for both first and second order finite elements.
For the pyramid inclusion, in Figure~\ref{fig:linsys_size}(b,d), FFTH outperforms FEM of order one but not of FEM of order two.

We note finally that the same observations are made for smaller phase contrast, for which we considered $\rho = 10$. For conciseness these results are not shown.

%\clearpage
\subsection{Memory requirements of linear systems}\label{sec:memory-requirements-of-linear-systems}

Here, the discretisation error is compared with the memory requirements for linear systems.
We note that the CG solver imposes some additionally memory requirements, which have been neglected here. This contribution is minor and similar for both methods and it only involves storing three more vectors.

For FEM, the memory requirements are determined by an unknown- and a right-hand side vector plus the storage of the system matrix assuming compressed sparse row (CSR) storage
\begin{equation}
\tag{FEM}
\nnz\MB{u} + \nnz\MB{b} + \underbrace{2\nnz \MBA + \rank\MBA}_{\text{storing $\MBA$ in CSR format}}
\end{equation}
(wherein the symmetry of $\MBA$ has been used, but the difference in the storage for floats, used for the components, and for integers, used for the indices, has been neglected).

For FFTH, no special sparse storage for the block diagonal $\MBA$ is needed, but additionally the Fourier coefficients of the projection matrix $\MBhG$ have to be stored, which results in
\begin{align*}
\tag{FFTH-Ga}
\nnz\MBe + \nnz\MB{b} + \nnz \MBhG + \nnz \MBA &= 2dN^d + d^2N^d + d^2(2N-1)^d,
\\
\tag{FFTH-GaNi}
\nnz\MBe + \nnz\MB{b} + \nnz \MBhG + \nnz \widetilde{\MBA} &= 2dN^d + d^2N^d + d^2N^d = (2d+2d^2)N^d.
\end{align*}
Although not pursued here, in principle, the projection $\MBhG$ could be evaluated just-in-time, avoiding its storage.

In the Figure~\ref{fig:linsys_memory}, we can see that the results in memory requirements are comparable to those of the size in linear system (in Figure~\ref{fig:linsys_size}).
FEM has lower memory requirements for square inclusion, in both 2D and 3D, even when the storage of the Cholesky factors in the direct solver is considered.
As reported in \cite{Vondrejc2015FFTimproved}, FFTH-Ga performs better than FFTH-GaNi for configurations with jumps in the material coefficients.
What was not yet observed is that the opposite result is obtained for more regular data (for the pyramid inclusion).
For this case, FFTH performs better than first order FEM and comparable to second order FEM.
This becomes even more pronounced when considering the additional memory requirements that are necessary for preconditioning the linear systems of FEM, see Section~\ref{sec:condition-numbers-of-linear-systems}.

\begin{figure}[htb]
 \centering
 \subfigure[\squar, $d=2$, $\rho=\phase$]{
  \includegraphics[width=0.46\textwidth]{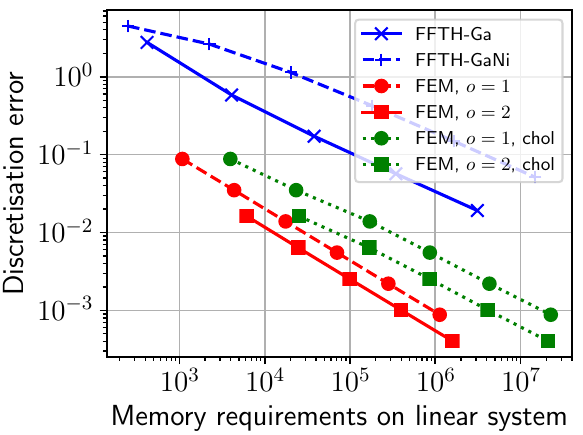}
 }
 \subfigure[\pyramid, $d=2$, $\rho=\phase$]{
  \includegraphics[width=0.46\textwidth]{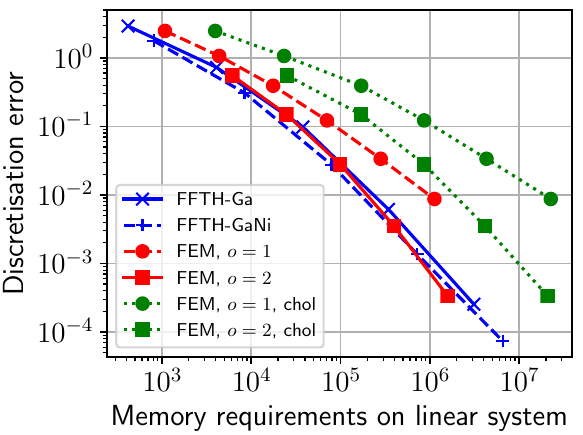}
 }
 \subfigure[\squar, $d=3$, $\rho=\phase$]{
  \includegraphics[width=0.46\textwidth]{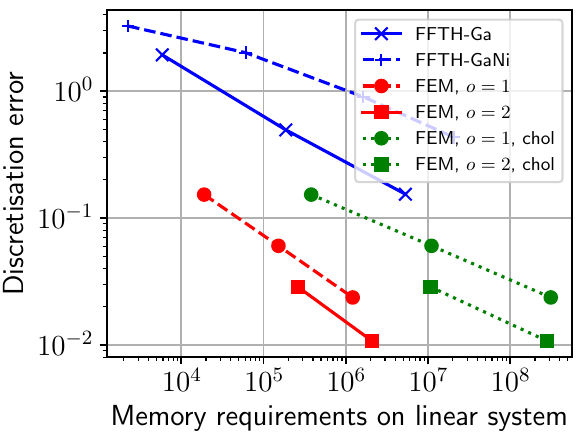}
 }
 \subfigure[\pyramid, $d=3$, $\rho=\phase$]{
  \includegraphics[width=0.46\textwidth]{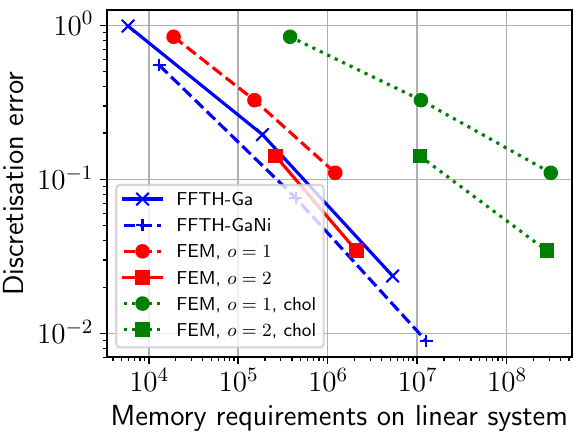}
 }
 \caption{Dependence of the approximate discretisation error in the energetic norm (see Eq.\ \eqref{eq:discretisation_error}) on the memory requirements of the linear systems (see Section~\ref{sec:memory-requirements-of-linear-systems}). The organisation of the figures is the same as in Figure~\ref{fig:linsys_size}, whereby the additional dotted green lines account for the extra memory requirements for the direct solver (i.e.\ for the storage of the Cholesky factors).}
 \label{fig:linsys_memory}
\end{figure}

% ==============================================================================
\subsection{Condition numbers of linear systems}
\label{sec:condition-numbers-of-linear-systems}
% ==============================================================================

\begin{figure}[htb]
	\centering
	\subfigure[\squar, $d=2$, $\rho=10$]{
		\includegraphics[width=0.46\textwidth]{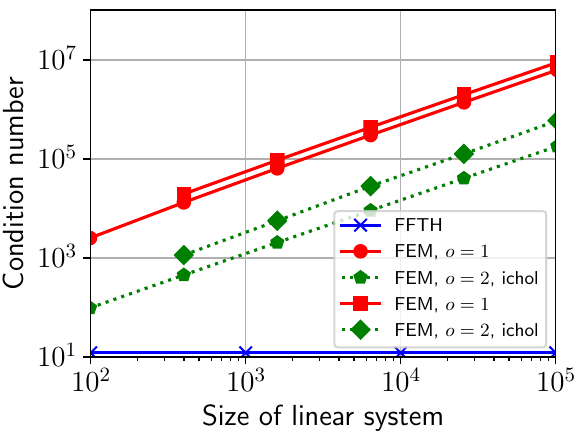}
	}
	\subfigure[\pyramid, $d=2$, $\rho=10$]{
		\includegraphics[width=0.46\textwidth]{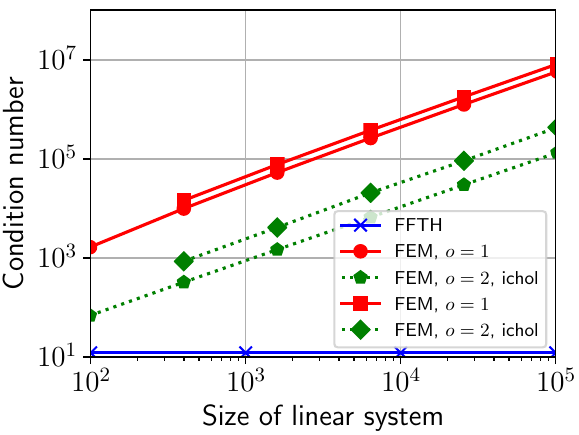}
	}
	\subfigure[\squar, $d=2$, $\rho=100$]{
		\includegraphics[width=0.46\textwidth]{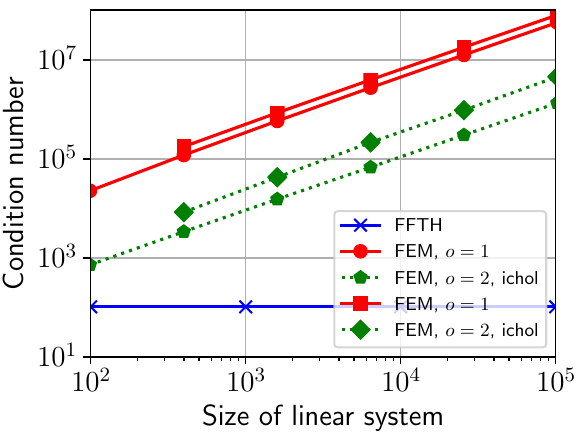}
	}
	\subfigure[\pyramid, $d=2$, $\rho=100$]{
		\includegraphics[width=0.46\textwidth]{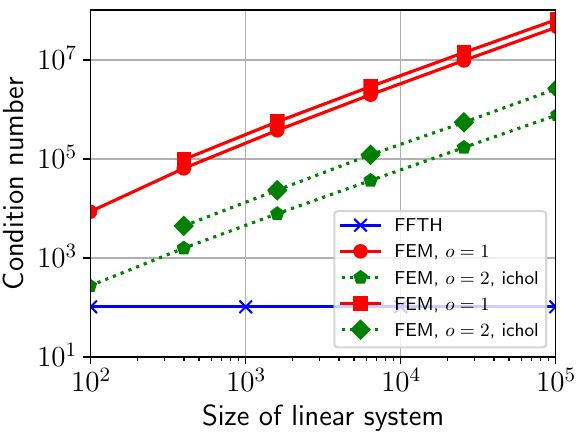}
	}
	\caption{Condition number of the linear systems in $d = 2$. The rows now correspond to different phase contrasts $\rho$, while the columns still correspond to the different considered shapes. The dotted green lines shows the condition numbers for the systems preconditioned with the no-fill variant of the incomplete Cholesky decomposition. The results for $d = 3$ are shown in Fig.~\ref{fig:ls_cond_d3}.}
	\label{fig:ls_cond_d2}
\end{figure}

\begin{figure}[htb]
	\centering
	\subfigure[\squar, $d=3$, $\rho=10$]{
		\includegraphics[width=0.46\textwidth]{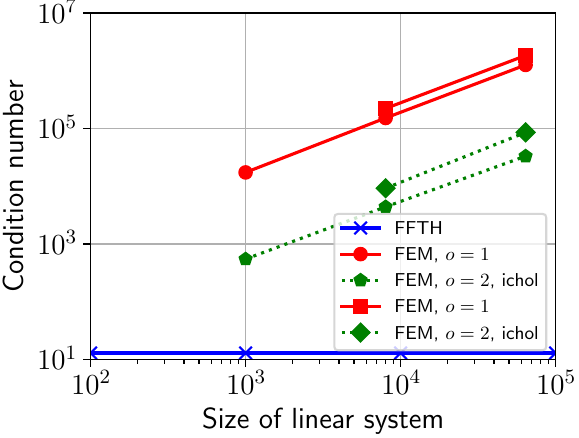}
	}
	\subfigure[\pyramid, $d=3$, $\rho=10$]{
		\includegraphics[width=0.46\textwidth]{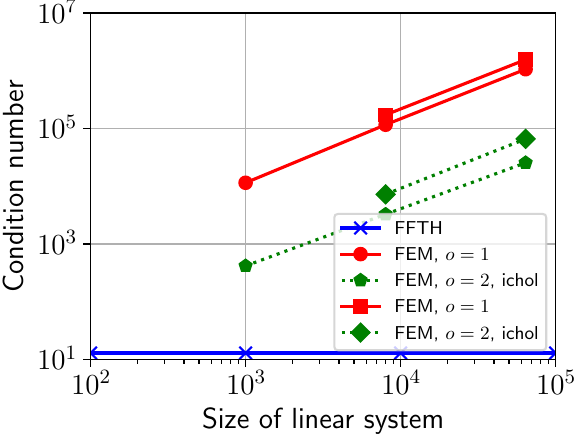}
	}
	\subfigure[\squar, $d=3$, $\rho=100$]{
		\includegraphics[width=0.46\textwidth]{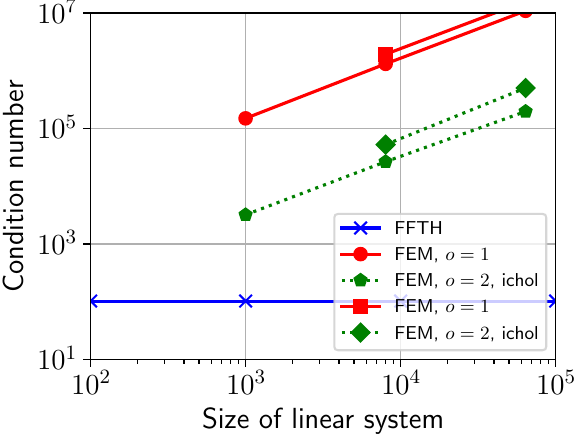}
	}
	\subfigure[\pyramid, $d=3$, $\rho=100$]{
		\includegraphics[width=0.46\textwidth]{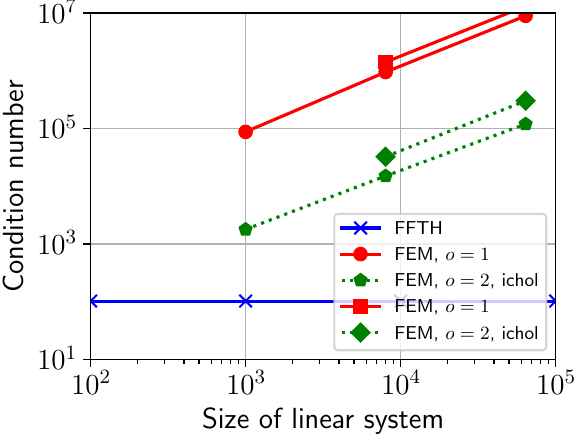}
	}
	\caption{Condition number of the linear systems in $d = 3$. The organisation of the panels is the same as in Figure~\ref{fig:ls_cond_d2} (for $d = 2$).}
	\label{fig:ls_cond_d3}
\end{figure}

The convergence of the CG solver depends on the condition number, particularly on the distribution of eigenvalues of the linear system.
Therefore, the condition numbers are depicted in Figures~\ref{fig:ls_cond_d2} and \ref{fig:ls_cond_d3} for different phase contrasts $\rho$, numbers of DOFs, and different dimensions (2D and 3D).
For FFTH, the condition number remains the same regardless the number of basis functions, thanks to their orthogonality, and depends only on the material contrast.
Particularly, it can be estimated for the square- and the pyramid-like inclusion using the smallest and the largest eigenvalues
\begin{equation*}
1=c_A \leq \lam_\mathrm{min} \leq \frac{\MBA\MBx\cdot\MBx}{\MBx\cdot\MBx}
\leq \lam_\mathrm{max} \leq  C_A=\rho+d
\quad\text{for all }\MBx\in\xEN,
\end{equation*}
which depends only on the ellipticity and continuity constants $c_A$ and $C_A$ of the material coefficients defined in \ref{sec:continuous-formulation} (see e.g. \cite{VoZeMa2014FFTH,MiVoZe2016jcp} for details).

The condition number for FEM, studied e.g.\ in \cite{Ern2006}, depends on the mesh size.
The increasing number of basis functions results in a deterioration of the condition number, which calls for the employment of special techniques such as preconditioning.
The resulting condition numbers of the different configurations, based on the no-fill variant of the incomplete Cholesky decomposition, are shown using green lines in Figures~\ref{fig:ls_cond_d2} and \ref{fig:ls_cond_d3}.
In spite of a significant improvement of the condition numbers, their dependence on the number of DOFs remains the same.
This can only be avoided by more advanced preconditioners.

%\clearpage
% ==============================================================================
\subsection{Computational effectiveness of the conjugate gradient solver}\label{sec:computational-effectiveness-of-the-conjugate-gradient-solver}
% ==============================================================================

\begin{figure}[htb]
	\centering
	\subfigure[\pyramid, $d=2$, $\rho=10$]{
		\includegraphics[width=0.46\textwidth]{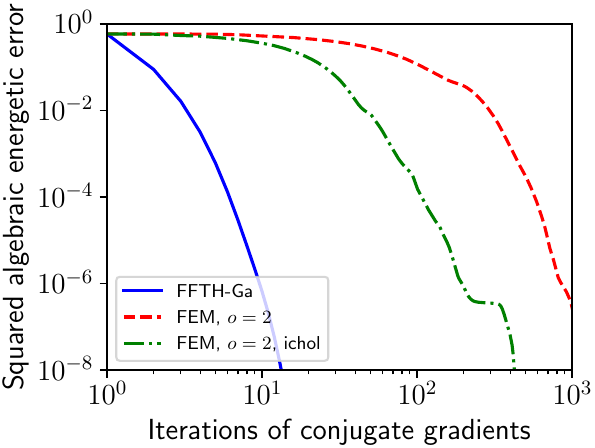}
	}
	\subfigure[\pyramid, $d=3$, $\rho=10$]{
		\includegraphics[width=0.46\textwidth]{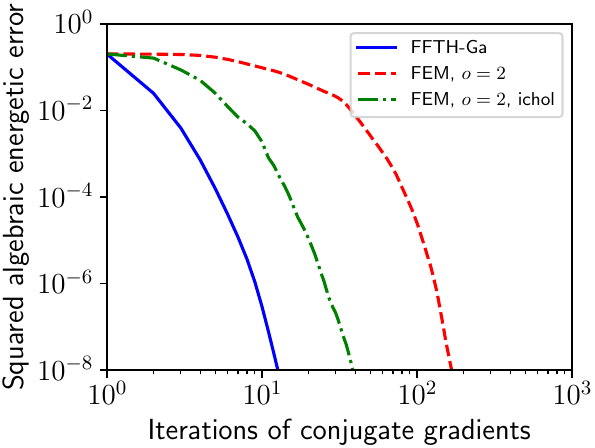}
	}
	\subfigure[\pyramid, $d=2$, $\rho=100$]{
		\includegraphics[width=0.46\textwidth]{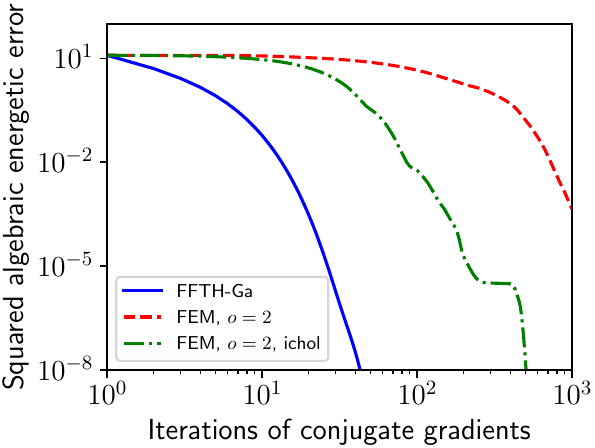}
	}
	\subfigure[\pyramid, $d=3$, $\rho=100$]{
		\includegraphics[width=0.46\textwidth]{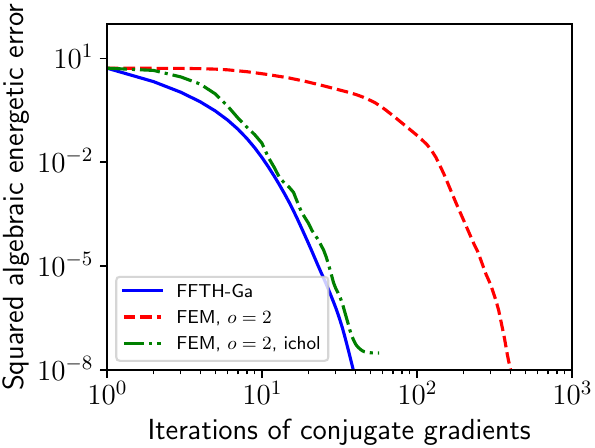}
	}
	\caption{The algebraic error during iterations of the conjugate gradients.
		For FFTH, the discretisation grid is of size $\VN=[405,405]$ in $d = 2$ and of size $\VN=[45,45,45]$ in $d = 3$.
		For FEM it is of size $\VN=[160,160]$ in $d = 2$ and of size $\VN=[20,20,20]$ in $d = 3$.
		The rows again correspond the different phase contrasts $\rho$, while this time the columns correspond to the number of dimensions; all figures are for the pyramid inclusion.}
	\label{fig:cg}
\end{figure}

The effectiveness of the conjugate gradient method is depicted in Figure~\ref{fig:cg} for both FFTH and FEM. This time we consider only one discretisation per dimension $d$ and phase contrast $\rho$.
The size of the linear systems are chosen to provide approximately the same discretisation error.
Since FFTH is not competitive for jumps in material coefficients, only the pyramid inclusion is considered.
It shows an \emph{algebraic energetic error} in iteration $k$, which also equals to an \emph{algebraic error in the upper bound}, i.e.\
\begin{equation*}
\|\MBe\iter{k}-\MBe\|_{\MBA}^2
= \MBA(\MBe\iter{k}-\MBe)\cdot (\MBe\iter{k}-\MBe)
= A_{\eff,\VN,(k)} - A_{\eff,\VN},
\end{equation*}
where $A_{\eff,\VN,(k)}$ represents the homogenised properties after the $k$-th iteration of the CG solver and $A_{\eff,\VN}$ represents the homogenised properties at convergence;
see \ref{lem:algebraic_error} for derivation.
Generally, the numerical results confirm better convergence in solving the linear system that arises from FFTH because of a better condition number.
For FEM, the preconditioner significantly reduces the number of CG iterations at reasonable memory requirements (see Figure~\ref{fig:linsys_memory}).
For $d = 3$ and the highest considered material contrast, FEM with preconditioner and FFTH-Ga achieve the similar behaviour.
However, FFTH-Ga still outperforms FEM for problems with a higher number of DOFs due to better conditioning of the linear system (results not shown).

The computational time for the CG solver is estimated by the product of the number of iterations and the cost for the matrix-vector multiplication, which is directed by the number of operations.
The matrix-vector multiplication for the FEM system is directed by the number of nonzero values in matrix $\MBA$. The computational cost of the matrix-vector multiplication for the FFTH system is directed by the FFT and its inverse, which requires $\frac{5}{2}|\VN|\log_2(|\VN|) = \frac{5}{2}N^d\log_2(N^d)$ operations (assuming real data on grid of size $\VN$).
The total number of operations for multiplication with $\MBG\MBA$ also requires the element-wise multiplication with matrices $\MBA$ and $\MBhG$, which leads to
\begin{alignat*}{3}
\tag{FFTH-Ga}
&\underbrace{d^2 (2N-1)^d}_{\text{operations for } \MBA} \;&&+ \underbrace{d^2 N^d}_{\text{operations for } \MBhG} &&+ \underbrace{5 d (2N-1)^d\log_2((2N-1)^d)}_{\text{operations for } \MBF \text{ and } \MBFi},
\\
\tag{FFTH-GaNi}
&\underbrace{d^2 N^d}_{\text{operations for } \widetilde{\MBA}} &&+ \underbrace{d^2 N^d}_{\text{operations for } \MBhG} &&+ \underbrace{5 d N^d\log_2(N^d)}_{\text{operations for } \MBF \text{ and } \MBFi}.
\end{alignat*}

\begin{figure}[htb]
 \centering
 \subfigure[\squar, $d=2$, $\rho=\phase$]{
  \includegraphics[width=0.46\textwidth]{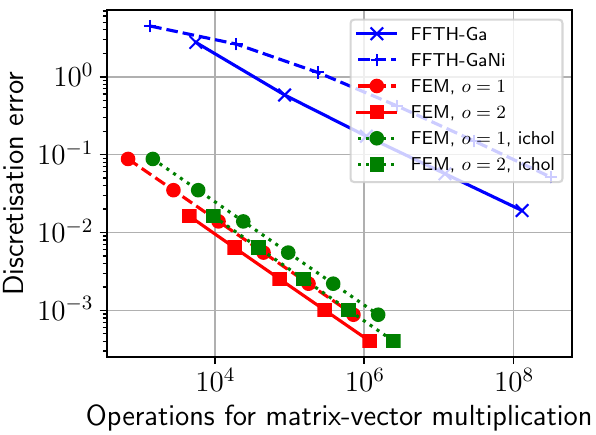}
 }
 \subfigure[\pyramid, $d=2$, $\rho=\phase$]{
  \includegraphics[width=0.46\textwidth]{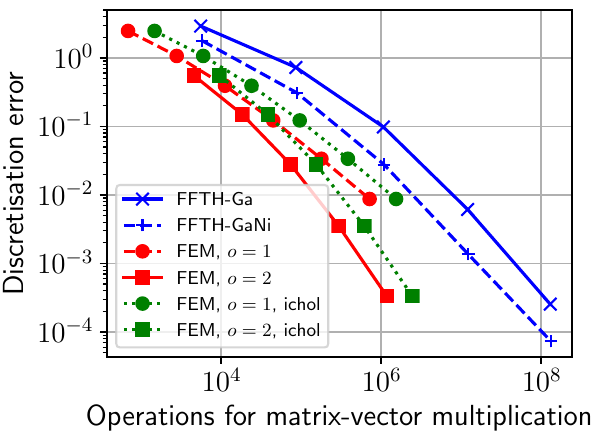}
 }
 \subfigure[\squar, $d=3$, $\rho=\phase$]{
  \includegraphics[width=0.46\textwidth]{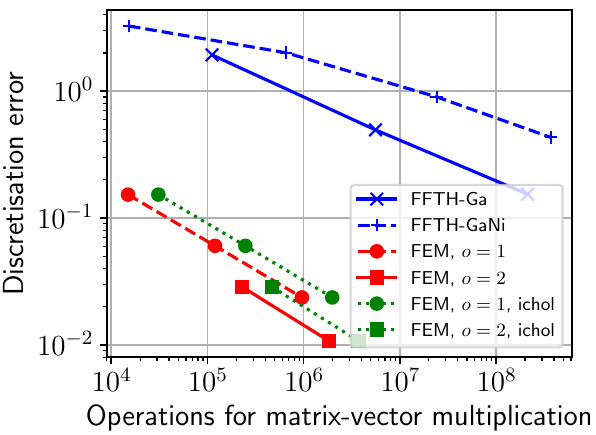}
 }
 \subfigure[\pyramid, $d=3$, $\rho=\phase$]{
  \includegraphics[width=0.46\textwidth]{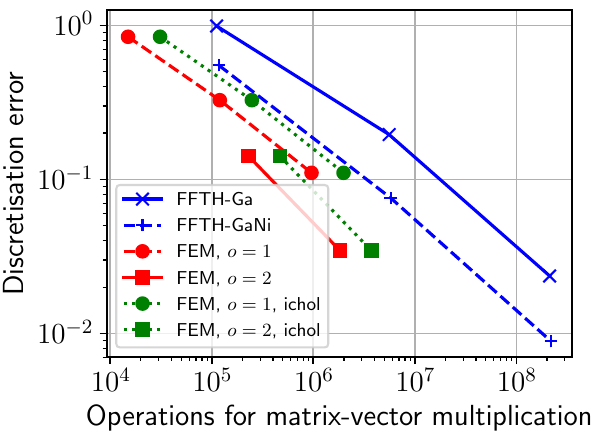}
 }
 \caption{Computational requirements determined by the number of operations for the matrix-vector multiplication. The organisation of the panels is the same as in Figure~\ref{fig:linsys_memory}.}
 \label{fig:linsys_comp}
\end{figure}

The cost of the single matrix-vector multiplication is plotted in Figure~\ref{fig:linsys_comp}, again with respect to the discretisation error, for the two shapes, in $d = 2$ and $d = 3$.
FFTH always has higher computational requirements for the matrix-vector multiplication than FEM, even when the operations for the preconditioner are considered.
The difference is more pronounced for the square inclusion and for the higher order polynomials in FEM.
To conclude, the lower computational requirements for the matrix-vector multiplication in FEM are balanced by a higher number of iterations needed for the CG solver.

%\clearpage
\subsection{Non-regular geometries}
\label{sec:examples:extra}

Although the previous examples are well controlled, they are not very realistic as many practical configurations are not formulated on regular grids.
In that case the difference between FEM and FFTH is more distinct, as any configuration can be accurately discretised using FEM by exploiting its capability of using conforming meshes.
In contrast, the discretisation is always through a regular grid for FFTH.

We therefore consider two more realistic examples. We commence with a circular inclusion, which has a smooth boundary. The material parameters are chosen in the same way as for the square inclusion in Section~\ref{sec:material-parameters}.
Its numerical treatment is a quite straightforward for FFTH-Ga because it easily allows exact (and also effective numerical) integration, for details see \cite[Section~6]{VoZeMa2014GarBounds} and \cite[Section~4.2]{Vondrejc2015FFTimproved}.

A more complicated situation arises for FEM, which cannot directly approximate the circle's boundary with linear triangles.
To construct the mesh, the boundary is approximated by polygonal domain.
For our comparison, we also require that the numerical approximation still produces the guaranteed bound on the homogenised properties and thus on the discretisation error based on the energetic norm (see Section~\ref{sec:discretisation-error}).
This is achieved by using an outer approximation of the boundary, proposed in \cite{Dvorak1993master} and depicted in Fig.~\ref{fig:circle_mesh} for a very coarse mesh;
all the elements that are at least partially inside the circle are considered to have material properties of the inclusion.

The discretisation error in Fig.~\ref{fig:circle_error} is significantly smaller for FEM than for FFTH for the same size of the linear system, which provides an analogical result as for the square inclusion.
To conclude, FEM provides better approximation for problems with jumps regardless the smoothness of the inclusion.

\begin{figure}[htb]
\centering
\subfigure[Mesh for circle inclusion]{
	\label{fig:circle_mesh}
\includegraphics[height=0.34\textwidth]{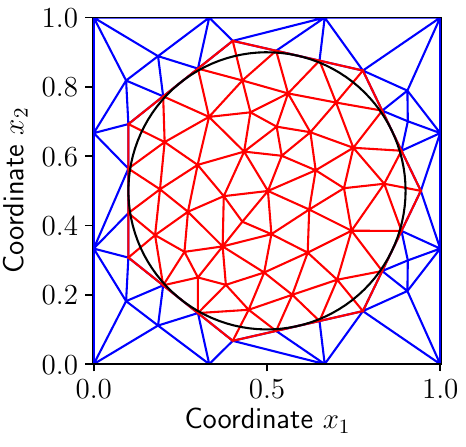}}
\subfigure[Discretisation errors]{
	\label{fig:circle_error}
\includegraphics[height=0.34\textwidth]{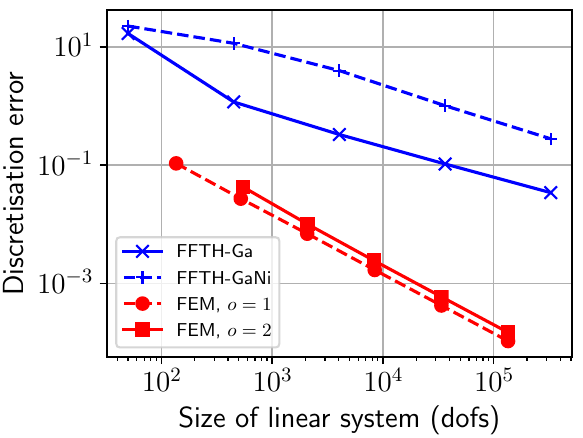}}
\caption{(a) Outer triangulation for the circle inclusion, (b) dependence of the approximate discretisation error in the energetic norm on the size of linear systems (see Eq.~\eqref{eq:discretisation_error}).}
\label{fig:circle_results}
\end{figure}

Finally, we investigate a less academic example which employs an image-based configuration.
This is an interesting case, because it is often argued that FFTH is ideally suited for such a setting, as the input is already discretised using a regular grid.
The fly ash-based aerated concrete, investigated in e.g. \cite{Hlavacek2014flyash}, has been considered here with a resolution $75\times 75 \times 75$, see Figure~\ref{fig:image_material}, and isotropic material properties $0.49\,\textrm{Wm}^{-1}\textrm{K}^{-1}$ for fly ash and $0.026\,\textrm{Wm}^{-1}\textrm{K}^{-1}$ for voids.
The results are presented in Table~\ref{tab:e} whereby the discretisation coincides with the image's resolution each time.

\begin{figure}
\centering
\includegraphics[height=0.34\linewidth]{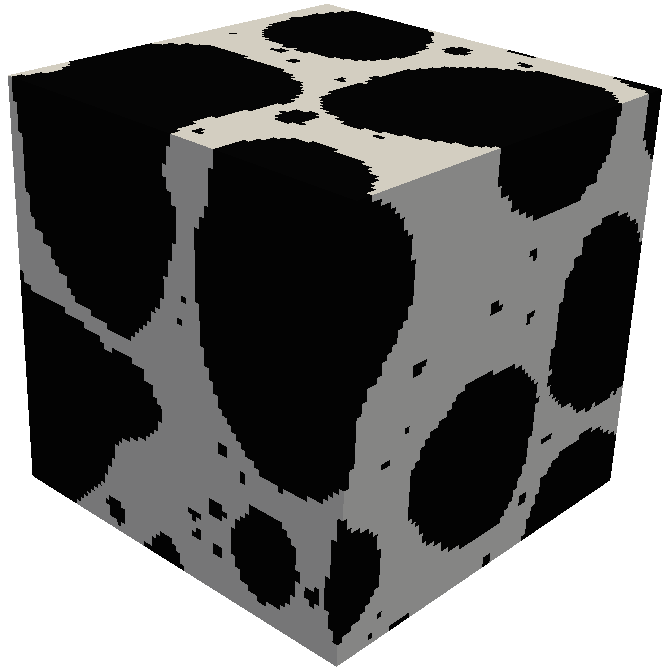}
\caption{Topology of image-based material with resolution $75\times 75\times 75$}
\label{fig:image_material}
\end{figure}

The results confirm the already presented results. FFTH-Ga has a higher discretisation error than FEM.
Moreover, FFTH has a three-times larger linear system than FEM of order $1$ because in FFTH one solves for the gradient field having three components in 3D.
FEM of order $2$ has the lowest discretisation error at the price of having a $2^3=8$ times larger linear system compared to FEM of order $1$.
FFTH-GaNi predicts the homogenised properties quite accurately. However the discretisation error is the highest.

\begin{table}[htb]
\centering
%\small
\begin{tabular}{c|ccccc}
%\toprule
\hline
& FFTH-Ga & FFTH-GaNi  & FFTH-GaNi bound  & FEM $o=1$ & FEM $o=2$ \\
&  Eq.~\eqref{eq:GA_minimisation} & Eq.~\eqref{eq:GaNi_minimisation} & Eq.~ \eqref{eq:gani_bound} & Eq.~\eqref{eq:FEM_minimisation} & Eq.~\eqref{eq:FEM_minimisation} \\
\hline
%\midrule
$A_H$ & $\leq0.13751$ & $\approx 0.13328$ & $\leq 0.14147$ & $\leq 0.13748$ & $\leq 0.13420$ \\
\#DOFs\ & $1.266\cdot 10^6$ & $1.266\cdot 10^6$ & $1.266\cdot 10^6$ & $4.219\cdot 10^5$ & $3.375\cdot 10^6$ \\
\hline
%\bottomrule
\end{tabular}
\caption{Homogenised properties of the fly ash-based aerated concrete specimen. The inequality signs signify that all the values are upper bound on the homogenised properties except for FFTH-GaNi, which provides an approximation of the homogenised properties.}
\label{tab:e}
\end{table}

\section{Conclusion and discussion}
\label{sec:conclusion}

In this paper, the finite element method (FEM) is compared to two versions of the Fourier-Galerkin method (FFTH): the Fourier-Galerkin method with numerical integration (FFTH-GaNi) \cite{ZeVoNoMa2010AFFTH,VoZeMa2014FFTH} that corresponds to the original Moulinec-Suquet algorithm, and an equivalent scheme based on exact integration (FFTH-Ga).
The latter scheme is more accurate than FFTH-GaNi but with higher memory and computational requirements \cite{Vondrejc2015FFTimproved}.
It also allows the evaluation of the energetic norm, which is used here as a criterion for comparison.
Particularly, the square norm of the local error directly also corresponds to the error in the homogenised properties, allowing us to, for the first time, make the comparison directly for local fields based on an objective criterion with a clear physical meaning.
The most important findings are summarised here:
\begin{itemize}
	\item FEM has better approximation properties than FFTH for rough data (with jumps in material coefficients). This results in smaller linear systems, lower memory requirements, and lower computational requirements than using FFTH to reach the same level of accuracy. This conclusion holds regardless of the smoothness of geometry (whereby the extreme square and circular geometries have been considered).

\item FFTH provides a good condition number of the linear systems, independent of the discretisation parameters.
The linear system that arises in FEM is less well conditioned, whereby the conditioning deteriorates for increasing system sizes.
For FFTH, this results in fewer iterations of the conjugate gradient (CG) solver used to solve the linear system.

\item The comparison for continuous material coefficients (pyramid inclusion) results in ambivalent conclusions.
The better conditioning of the FFTH linear system is compensated by lower computational costs per CG iteration in FEM (as this is only directed by the number of operations for the matrix-vector multiplication).
FFTH has comparable memory requirements as second order FEM and thus even lower when some preconditioner is used in FEM (which can be avoided in FFTH because of the favourable conditioning of the linear system).

\item As reported in \cite{Vondrejc2015FFTimproved}, FFTH-Ga, using exact integration, is better than FFTH-GaNi, using numerical integration, for materials with jumps.
What was not yet observed is that the opposite result is obtained for more regular data (the pyramid inclusion).
\end{itemize}

Our comparison also has a series of limitations, which are briefly discussed here.
\begin{itemize}
	\item An interesting comparison could also arise for \emph{dual problems} --- formulated for flux or stress.
	This would be beneficial for FFTH as it treats the primal and dual problems equally. On the contrary, mixed formulations are required by FEM which have higher memory and computational requirements and call for special solvers.
	\item We have omitted more \emph{advanced numerical techniques} such as adaptivity, multigrid, or domain decomposition methods, which are well developed for FEM but not for FFTH.
	\item Two versions of Fourier-Galerkin method (FFTH-Ga and FFTH-GaNi) have been considered. The comparison for \emph{other FFT-based approaches} (e.g.\ \cite{Brisard2012FFT,Willot2015,Schneider2016hexa}) would also be valuable, especially their approximation properties, conditioning of the system, etc.
	\item Further studies have to investigate more complex, \emph{non-linear problems}.
\end{itemize}

\subsection*{Acknowledgement}
This paper was supported by the Deutsche Forschungsgemeinschaft (grant number MA 2236/27-1) and by the Czech Science Foundation (grant number 17-04150J).
Tom de Geus was financially supported by The Netherlands Organisation for Scientific Research (NWO) by a NWO Rubicon grant number 680-50-1520.

%==========================
\appendix
%\label{}
%==========================
\section{Continuous formulation}
\label{sec:continuous-formulation}
We introduce here continuous formulation of the homogenisation problem, briefly introduced already in Section~\ref{sec:problem-description}. It is followed by the finite element discretisation in \ref{sec:finite-element-method} and the equivalent for the Fourier-Galerkin method in \ref{sec:fourier-galerkin-method}.

%================
Here, and in the sequel, the measurable matrix field $\B{A}:\puc\rightarrow\sR^{d\times d}$ is reserved for the material coefficients which are required to be essentially
bounded, symmetric, and uniformly elliptic. This means that for almost all $\x\in\puc$, there are constants $0 < c_A \leq C_A < +\infty$ such that
\begin{align}
%\tag{A}
\label{eq:A}
\B{A}(\x)&=\B{A}^T(\x),
&
c_A \norm{\Vv}^2
&\leq
\B{A}(\x)\Vv\cdot\Vv
\leq C_A \norm{\Vv}^2
\quad\text{ for all }\Vv \in \sR^d.
\end{align}
Employing the material coefficients, we systematically employ the bilinear form $\bilf{}{}:L^2(\puc;\sR^d)\times L^2(\puc;\sR^d)\rightarrow\sR$
defined as
\begin{align*}
%\label{eq:def_bilinear_forms}
\bilf{\Vv}{\Vw} &:
=\int_\puc \B{A}(\x)\Vv(\x) \cdot \Vw(\x)\D{\x}.
\end{align*}
%-----------------------------------

%\begin{definition}[Homogenisation problem]
%\label{def:homog_problem}

Then, the homogenisation problem is equivalently formulated in two ways: find the homogenised matrix $\B{A}_\eff\in\sR^{d\times d}$ satisfying the following minimisation problem for any macroscopic vector $\VE\in\sR^d$
\begin{subequations}
	\label{eq:homog_problem}
	\begin{align}
	\label{eq:homog_FEM}
	\B{A}_\eff\VE\cdot\VE  &= \min_{v\in H^1_0(\puc)}
	\bilf{\VE+\nabla v}{\VE+\nabla v}=\bilf{\VE+\nabla u}{\VE+\nabla u},
	\\
	\label{eq:homog_FFTH}
	&=
	\min_{\Vv\in\cE}
	\bilf{\VE+\Vv}{\VE+\Vv}=\bilf{\VE+\Ve}{\VE+\Ve},
	\end{align}
\end{subequations}
where we minimise over the scalar-valued Sobolev functions or directly over their gradients. For the latter case the minimisation space is
\begin{align*}
\cE &= \nabla \Hnp(\puc) = \{\Vv\in L^2(\puc;\sR^d):\curl\Vv = \Vo,\mean{\Vv}=\Vo\}.
\end{align*}
Note that we systematically use the variables $\Ve=\nabla u$ for the minimiser and $v$ or $\Vv$ for the scalar valued or vector valued test functions.

The first formulation \eqref{eq:homog_FEM} is suitable for the finite element method, while the second one \eqref{eq:homog_FFTH} is suitable for the Fourier-Galerkin method. In the latter case, a curl-free and a zero-mean conditions are easily enforced using an orthogonal projection $\mathcal{G}:L^2(\puc;\sR^d)\rightarrow\cE$, defined as
\begin{align}
\label{eq:proj_continuous}
\mathcal{G}[\Vv] = \sum_{\Vk\in\sZ^d} \frac{\Vk\otimes\Vk}{\Vk\cdot\Vk} \widehat{\Vv}(\Vk)\varphi^\Vk,
\qquad\text{ with }\varphi^\Vk(\x) = \exp(2\pi\imu \Vk\cdot\x)
\end{align}
for details see \cite[Lemma~2]{VoZeMa2014FFTH}.

To ease the comparison between the two discretisation methods, we will compare only one component of the homogenised matrix.
Therefore we fix the macroscopic vector to be $\VE=(\delta_{1i})_{i=1}^d\in\sR^d$ and consider the scalar effective energy
\begin{align*}
\B{A}_\eff\VE\cdot\VE = \B{A}_{\eff,11} = A_{\eff}.
\end{align*}

%=================================
\section{Finite element method}
\label{sec:finite-element-method}
%=================================
The finite element method is a numerical discretisation method of PDEs based on an approximation with piece-wise polynomials defined on finite elements.
The finite element mesh on $\puc$, denoted as $\mesh(\puc)$, will be composed of simplexes (triangles or tetrahedra) regularly placed in the computational domain $\puc$ as in Figure~\ref{fig:mesh};
the number of elements will be denoted with $\VN=[N,\dotsc,N]\in\sN^d$, which corresponds to a uniform discretisation in each spatial dimension.
%the total number of elements is then $2^dN^d$.

The finite element space consists of continuous functions that are polynomials of order $p$ on each finite element.
This is explicitly defined as
\begin{align*}
\FE &= \Big\{ v\in L^2(\puc)\bigcap C(\puc): v|_T\in \cP_p(T) \text{ for every }T\in\mesh(\puc) \Big\},
\end{align*}
where the space of polynomials of order $p$ is
$$\cP_p = \Big\{ f:\sR^d\rightarrow\sR \spr f(\x) = \sum_{\V{\alp}\in\sN^d_0,\sum_i \alpha_i\leq p} a_{\V{\alpha}} x_1^{\alpha_1}x_2^{\alpha_2}\dotsc x_d^{\alpha_d}, a_{\V{\alpha}} \in \sR \Big\}.$$

Since only the gradient of the state variable occurs in the variational formulation as in the Neumann problem, we have to fix one degree of freedom in order to obtain unique solution.
Here we fix the functional values at the corner of the cell
\begin{align*}
\FEo &= \{v\in\FE: v(\Vo)=0\}.
\end{align*}
Irrespective of the point and the value fixed, the gradient fields are always zero mean. This provides a conforming approximation to the homogenisation problem.

The FEM formulation of the homogenisation problem \eqref{eq:homog_FEM} states: find approximate homogenised coefficient $\Aeffem\in \sR$ satisfying
\begin{align}
\label{eq:FEM_minimisation}
\Aeffem  &=
\min_{v_{h,p}\in\FEo}
\bilf{\VE+\nabla v_{\VN,p}}{\VE+\nabla v_{\VN,p}} = \bilf{\VE+\nabla u_{\VN,p}}{\VE+\nabla u_{\VN,p}}.
\end{align}

In order to solve this problem, we use the basis functions $\psi_i:\puc\rightarrow\sR$ for $i=1,\dotsc,n$ of the finite element space $\FEo$;
the number of basis functions corresponds to the dimension of the FEM space  $n = \dim\FEo = \bigl( Np \bigr)^d-1$.
The minimiser $u_{\VN,p}\in\FEo$, expressed with respect to this basis
\begin{align*}
u_{\VN,p} = \sum_{i=1}^n \MB{u}_{i}\psi_i,
\end{align*}
is then characterised by the optimality condition (weak formulation)
\begin{align*}
\bilf{\VE+\nabla u_{\VN,p}}{\nabla v} = 0\quad\forall v\in\FEo
\end{align*}
or by the linear system for $\MB{u}\in\sR^n$
\begin{align}
\label{eq:linsys_FEM}
\MBA\MB{u}&= \MB{b},
&
\MBA_{ij} &= \bilf{\nabla\psi_j}{\nabla\psi_i},
&
\MB{b}_{i} &= -\bilf{\VE}{\nabla\psi_i}.
\end{align}

%================
\section{Fourier-Galerkin method}
\label{sec:fourier-galerkin-method}
%================

The Fourier-Galerkin method builds on the discretisation with trigonometric polynomials that are well defined on regular grids of size $\VN \in \sN^d$, for details see e.g. \cite{SaVa2000PIaPDE,Vondrejc2015FFTimproved}.
Before we proceed, we will briefly introduce the following notation for vectors and a matrices
\begin{align*}
\MB{a}_{\VN}
=
\left(
a_\alpha^\Vk
\right)_{\alpha=1,\ldots,d}^{\Vk \in \ZNd}
\in \sR^{d\times\VN},
&&
\MB{A}_{\VN}
=
\left(
A_{\alpha\beta}^{\Vk\Vm}
\right)_{\alpha,\beta=1,\ldots,d}^{\Vk,\Vm \in \ZNd}
\in [\sR^{d\times\VN}]^2,
\end{align*}
where the vectors $\Vk$ belongs to the index space
\begin{align*}
\ZNd  =
\biggl\{ \Vk \in \set{Z}^d :
|k_\alpha| < \frac{N_\alpha}{2} \biggr\}.
\end{align*}
The corresponding matrix-vector and matrix-matrix multiplication are understood as
\begin{align*}
(\MB{A}_{\VN}\MB{a}_{\VN})^{\Vk}_\alp
=
\M{A}^{\Vk\Vn}_{\alp\beta}\M{a}^{\Vn}_\beta,
&&
(\MB{A}_{\VN} \MB{B}_{\VN})^{\Vk\Vm}_{\alp\gamma}
=
\MB{A}^{\Vk\Vn}_{\alp\beta}\M{B}^{\Vn\Vm}_{\beta\gamma}.
\end{align*}

The space $\sR^{d\times\VN}$ is endowed with the following inner product and norm
\begin{align*}
\MB{a}_{\VN}\cdot\MB{b}_{\VN}
=
\frac{\M{a}^\Vk_\alp \M{b}^\Vk_\alp}{\pVN}
,
&&
\| \MB{a}_{\VN} \|^2 = \MB{a}_{\VN}\cdot\MB{a}_{\VN},
\end{align*}
where $\pVN = \prod_{\alp=1}^{d} N_\alp$ is the normalisation factor satisfying that the norm on $\sR^{d\times\VN}$ corresponds to the $L^2$-norm of trigonometric polynomials. The analogical notation is used for complex-valued quantities.

The approximation space consists of trigonometric polynomials
\begin{align*}
%\label{eq:space_trig_pol}
\cTNd &=
\Bigl\{\Vv_\VN=\sum_{\Vk\in\ZNd}{\MBv_\VN^{\Vk}\varphi^{ \Vk }_\VN} \spr \MBv^{\Vk}_\VN\in\sR^d\Bigr\},
\end{align*}
where the number of discretisation points in each direction $N$, is taken odd valued only;
the Nyquist frequencies that appear when $N$ is even have to be omitted to obtain a conforming approximation, see \cite{VoZeMa2014GarBounds} for details.
The basis functions $\varphi^{\Vk}_\VN:\puc\rightarrow\sR$, the so-called fundamental trigonometric polynomials, are expressed as a linear combination
\begin{align*}
\varphi_\VN^\Vk(\x)
=
\frac{1}{\pVN}
\sum_{\Vm \in \ZNd}
\omega_{\VN}^{-\Vk\Vm} \varphi^{\Vm}(\x)
\text{ for }
\x \in \puc,
\end{align*}
of Fourier polynomials $\varphi^{\Vm}(\x) = \exp\left(2\pi\imu\sum_{\alp} m_\alp x_\alp\right)$, which are well known from the Fourier series. The complex-valued numbers
$\omega_{\VN}^{\Vm\Vk} =
\exp   \bigl(2 \pi \imu\sum_{\alp} \frac{m_\alp k_\alp}{N_\alp}  \bigr)$
for $\Vm,\Vk\in \ZNd$ are coefficients of the discrete Fourier transform (DFT) or its inverse, particularly defined as
\begin{align}
\label{eq:DFT}
\MBF_{\VN} = \frac{1}{\pVN} \bigl( \del_{\alp\beta} \omega_{\VN}^{-\Vm\Vk} \bigr)_{\alp,\beta=1,\dotsc,d}^{\Vm,\Vk\in\ZNd} \in\xhMN ,
&&
\MBFi_{\VN} = \bigl( \del_{\alp\beta} \omega_{\VN}^{\Vm\Vk} \bigr)_{\alp,\beta=1,\dotsc,d}^{\Vm,\Vk\in\ZNd} \in\xhMN.
\end{align}

Thanks to the Dirac delta property of the fundamental trigonometric polynomials on a regular grid of points $(\x_\VN^\Vk=\frac{k_\alp Y_\alp}{N_\alp})_{\alp=1,\dotsc,d}$ for $\Vk\in\ZNd$, the coefficients of the trigonometric polynomials are equal to the functional values on the grid points, i.e.\
\begin{align*}
\Vv_\VN &= \sum_{\Vk\in\ZNd}{\MBv_\VN^{\Vk}\varphi^{ \Vk }_\VN},
&
\MBv_\VN^\Vk = \Vv_\VN (\x_\VN^\Vk).
\end{align*}

The space of trigonometric polynomials cannot be directly used as an approximation space because the curl-free and zero-mean conditions have to be enforced;
it can be easily provided by the continuous projection \eqref{eq:proj_continuous} to obtain the conforming space
\begin{align*}
\cEN = \cTNd\bigcap\cE = \mathcal{G}[\cTNd].
\end{align*}

%=======================================
\subsection{Fourier-Galerkin method with exact integration (FFTH-Ga)}\label{sec:FFTHga}
%=======================================
The Fourier-Galerkin method with exact integration (FFTH-Ga) for the homogenisation problem \eqref{eq:homog_FFTH} states: find the approximate homogenised coefficient $\Aeffth\in \sR$ satisfying
\begin{align}
\label{eq:GA_minimisation}
\Aeffth &= \inf_{\Vv_\VN\in\cEN} \bilf{\VE+\Vv_\VN}{\VE+\Vv_\VN} = \bilf{\VE+\Ve_\VN}{\VE+\Ve_\VN}.
\end{align}

The direct integration of the Fourier-Galerkin formulation, contrary to FEM, leads to a full-linear system, which can be overcome with a numerical integration on a double grid \cite{VoZeMa2014GarBounds,Vondrejc2015FFTimproved,Vondrejc2017DoGIP-FEM}, whereby
the original full-linear system of size $d|\VN|$ is reformulated to the block diagonal one of size $d|2\VN-\V{1}|$.
The minimiser with respect to the basis functions
\begin{align*}
\Ve_\VN=\sum_{\Vk\in\ZNd}{\MBe_\VN^{\Vk}\varphi^{ \Vk }_\VN}
\end{align*}
can be again found from the optimality condition (weak formulation) or its discrete version
\begin{align}
\label{eq:GA_weak}
\bilf{\VE+\Ve_\VN}{\Vv_\VN} &= 0
\quad \forall\Vv_\VN\in\cEN,
&
\scal{\MBA(\MBE+\MBe_\VN)}{\MBv_\VN}_{\sR^{d\times\tVN}}&= 0
\quad \forall\MBv_\VN\in\xEN,
\end{align}
where the discrete space corresponds to curl-free and zero-mean polynomials evaluated on a double grid
\begin{align}
\label{eq:xEtN}
\xEN = \left\{ \MBe = \bigl(\Ve_\VN(\x_\tVN^\Vk)\bigr)^{\Vk\in\ZtNd}\spr \Ve_\VN\in\cEN \right\} \subset\sR^{d\times\tVN}.
\end{align}
The vectors and the matrix coefficients are explicitly defined as
\begin{align*}
\MBA^{\Vk\Vl}_{\alp\beta} &= \del_{\Vk\Vl} \int_\puc A_{\alp\beta}(\x)\varphi_\tVN^\Vk(\x)\D{\x},
&
\MBe_\VN^\Vk &= \Ve_{\VN} (\x_\tVN^\Vk),
&
\MBv_\VN^\Vk &= \Vv_{\VN} (\x_\tVN^\Vk),
&
\MBE^\Vk &= \VE
\end{align*}
for $\Vk,\Vl\in\ZtNd$;
the closed-form evaluation of the block-diagonal matrix $\MBA\in\xMtN$ is described in \cite{VoZeMa2014GarBounds,Vondrejc2015FFTimproved}.

The linear system cannot be directly derived from the discrete weak formulation \eqref{eq:GA_weak} because the test vectors $\MBv\in\xEN$ do not span the whole space $\sR^{d\times\tVN}$ as test vectors corresponds to trigonometric polynomials with a curl-free and a zero-mean condition.
However, this condition can be enforced with the discrete orthogonal projection $\MBG:\xVtN\rightarrow\xEN$, derived from the continuous one \eqref{eq:proj_continuous}, i.e
\begin{align*}
\MBG &= \MBFi_\tVN \MBhG \MBF_\tVN,
&
\MBhG\in\xVtN
\text{ with components }
\MBhG^{\Vk\Vl} &=
\begin{cases}
\frac{\Vk\otimes\Vk}{\Vk\cdot\Vk} \del_{\Vk\Vl},
&
\text{for }\Vk,\Vl\in\ZNd
\\
\Vo
&
\text{for }\Vk,\Vl\in\ZtNd\setminus\ZNd
\end{cases},
\end{align*}
and $\MBF^{-1}_\tVN,\MBF_\tVN$ are the (inverse) DFT matrices from \eqref{eq:DFT}.
Then the linear system states: find $\MBe\in\xEN\subset\sR^{d\times\tVN}$ satisfying
\begin{align}
\label{eq:linsys_FFTHga}
\MBG\MBA\MBe_\VN = -\MBG\MBA\MBE,
\end{align}
which can be efficiently solved by conjugate gradients (for comparison of linear solvers see \cite{Moulinec2014comparison,MiVoZe2016jcp}). This system of size
%\begin{subequations}
%	\label{eq:dofs_ffth}
\begin{align*}
%\label{eq:dofs_ffth_all}
\dim\cT_{\tVN}^d &= d\cdot \prod_\alp (2N_\alp-1) = d(2N-1)^d,
\end{align*}
was also reformulated in \cite[section~4.4]{Vondrejc2015FFTimproved} to a reduced size $\dim\cTNd = dN^d$ with a very similar structure.
We also note that the number of independent degrees of freedom are governed by the dimension of gradient-valued trigonometric polynomials, expressed as
\begin{align*}
%\label{eq:dofs_independent}
\dim\cEN &= N^d-1.
\end{align*}
%\end{subequations}

\subsection{Fourier-Galerkin method with numerical integration (FFTH-GaNi)}
\label{sec:FFTHgani}

In previous section, FFTH-Ga based on exact integration has been presented.
Although, it is possible for a big class of material coefficients, it leads to higher computational and memory requirements as it is evaluated on a double grid.
A simpler approach incorporates the numerical integration of the bilinear forms
\begin{align*}
\bilf{\VE+\Ve_\VN}{\VE+\Ve_\VN} \approx \sum_{\Vk\in\ZNd} \B{A}(\x_\VN^\Vk)[\VE+\Ve_\VN(\x_\VN^\Vk)]\cdot[\VE+\Ve_\VN(\x_\VN^\Vk)]
=\widetilde{\MBA}(\MBE+\MBe_\VN)\cdot (\MBE+\MBe_\VN).
\end{align*}
This numerical scheme is fully equivalent to the original Moulinec and Suquet algorithm \cite{Moulinec1994FFT} as the solution vectors are the same \cite{VoZeMa2012LNSC,VoZeMa2014FFTH}.

The Fourier-Galerkin method with numerical integration (FFTH-GaNi) for the homogenisation problem \eqref{eq:homog_FFTH} states: find the approximate homogenised coefficient $\widetilde{A}_{\eff,\VN} \in \sR$ satisfying
\begin{align}
\label{eq:GaNi_minimisation}
\widetilde{A}_{\eff,\VN} &= \inf_{\MBe_\VN\in\xEN} \widetilde{\MBA}(\MBE+\MBe_\VN)\cdot (\MBE+\MBe_\VN) = \widetilde{\MBA}(\MBE+\widetilde{\MBe}_\VN)\cdot (\MBE+\widetilde{\MBe}_\VN),
\end{align}
where the discrete vectors $\MBv_\VN,\widetilde{\MBe}_\VN$ are, contrary to \eqref{eq:xEtN}, defined on an original grid
\begin{align}
\xEN = \left\{ \MBe_\VN = \bigl(\Ve_\VN(\x_\VN^\Vk)\bigr)^{\Vk\in\ZNd}\spr \Ve_\VN\in\cEN \right\} \subset\sR^{d\times\VN}.
\end{align}
The linear system corresponding to the optimality condition of \eqref{eq:GaNi_minimisation} has the same structure as in the case of FFTH-Ga. It states: find $\widetilde{\MBe}_\VN\in\xEN\subset\sR^{d\times\VN}$ satisfying
\begin{align}
%\label{eq:linsys_FFTH}
\MBG\widetilde{\MBA}\widetilde{\MBe}_\VN = -\MBG\widetilde{\MBA}\MBE.
\end{align}
Compared to FFTH-Ga system \eqref{eq:linsys_FFTHga}, the material matrix $\widetilde{\MBA}$ is changed and the discrete orthogonal projection $\MBG:\sR^{d\times\VN}\rightarrow\xEN$, derived from the continuous one \eqref{eq:proj_continuous}, reads
\begin{align*}
\MBG &= \MBFi_\VN \MBhG \MBF_\VN,
&
\MBhG\in\sR^{d\times\VN}
\text{ with components }
\MBhG^{\Vk\Vl} &=
\frac{\Vk\otimes\Vk}{\Vk\cdot\Vk} \del_{\Vk\Vl}
\quad
\text{for }\Vk,\Vl\in\ZNd
\end{align*}
The and $\MBF^{-1}_\VN,\MBF_\VN$ are the (inverse) DFT matrices from \eqref{eq:DFT}.

Because of inconsistency error caused by numerical integration, the homogenised value $\widetilde{A}_{\eff,\VN}$ does not provide a bound on $\Aeff$.
However, the discrete minimiser $\widetilde{\MBe}_\VN$ still represents a conforming trigonometric polynomial
\begin{align*}
\widetilde{\Ve}_\VN =	\sum_{\Vk\in\ZNd} \widetilde{\MBe}_\VN^\Vk\varphi_\VN^\Vk,
\end{align*}
which can be used to evaluate the bilinear forms a posteriori.
This gives rise to yet another homogenised value $\ol{A}_{\eff,\VN}$ complying with the structure of homogenised properties
\begin{align}
\label{eq:gani_bound}
\Aeff \leq A_{\eff,\VN} \leq \ol{A}_{\eff,\VN} := \bilf{\VE+\widetilde{\Ve}_\VN}{\VE+\widetilde{\Ve}_\VN},
\end{align}
see \cite{Vondrejc2015FFTimproved} for details.

\section{Relation between energetic norm and homogenised properties}
\label{sec:energetic-norm}

In this section, we make a relation between homogenised properties and energetic norm $\|\cdot\|_{\B{A}}:L^2(\puc;\sR^d)\rightarrow\sR$ that is defined as $\|\Vv\|^2_{\B{A}}=\bilf{\Vv}{\Vv}$.
%----------------------
\begin{lemma}
	\label{lem:ener_norm}
	Let $\mathcal{X}$ be a subspace of $L^2(\puc;\sR^d)$ and let $\Ve_\mathcal{X}$ be a minimiser of homogenisation problem over $\mathcal{X}$, i.e.\
	\begin{align}
	\label{eq:homog_X}
	\Ve_\mathcal{X} = \argmin_{\Vv\in\mathcal{X}} \bilf{\VE+\Vv}{\VE+\Vv}.
	\end{align}
	Then for any conforming vector $\Vv\in \mathcal{X}$, we have the relation
	between the square of the energetic norm of the error and the error in the homogenised properties
	\begin{align}
	\label{eq:relation_general}
	\|\Ve_\mathcal{X}-\Vv\|^2_{\B{A}} = \bilf{\VE+\Vv}{\VE+\Vv} - \bilf{\VE+\Ve_\mathcal{X}}{\VE+\Ve_\mathcal{X}}
	\end{align}
\end{lemma}
%.......................
\begin{proof} Using linearity of bilinear forms, we deduce
	\begin{align*}
	\|\Ve_\mathcal{X}-\Vv\|^2_{\B{A}} &= \bilf{\Ve_\mathcal{X}-\Vv}{\Ve_\mathcal{X}-\Vv}
	= \bilf{\Ve_\mathcal{X}-\Vv\pm\VE}{\Ve_\mathcal{X}-\Vv\pm\VE},
	\\
	&= \bilf{\VE+\Ve_\mathcal{X}}{\VE+\Ve_\mathcal{X}} + \bilf{\VE+\Vv}{\VE+\Vv} - 2\bilf{\VE+\Ve_\mathcal{X}}{\VE+\Vv}.
	\end{align*}
	Now, we reformulate the last term using the optimality condition (weak formulation) of \eqref{eq:homog_X}, namely $\bilf{\VE+\Ve_\mathcal{X}}{\Vv}=0\quad\forall \Vv\in\mathcal{X}$, which surely holds for a special choice of the test function as $\Vv=\Ve_\mathcal{X}$. Hence
	\begin{align*}
	\bilf{\VE+\Ve_\mathcal{X}}{\VE+\Vv} &= \bilf{\VE+\Ve_\mathcal{X}}{\VE} + \underbrace{\bilf{\VE+\Ve_\mathcal{X}}{\Vv}}_{=0},
	\\
	&= \bilf{\VE+\Ve_\mathcal{X}}{\VE} + \underbrace{\bilf{\VE+\Ve_\mathcal{X}}{\Ve_\mathcal{X}}}_{=0} = \bilf{\VE+\Ve_\mathcal{X}}{\VE+\Ve_\mathcal{X}}.
	\end{align*}
	The combination of two formulas in the proof gives as the required formula \eqref{eq:relation_general}.
\end{proof}
%-------------------
\begin{corollary}
	The squared energetic norm of discretisation error is equal to the error in the homogenised properties for the FEM, FFTH-Ga, and FFTH-GaNi discretisations, particularly
	\begin{align*}
	\|\Ve-\Ve_\VN\|^2_{\B{A}}
	&= \Aeffth - \Aeff,
	&
	\|\Ve-\widetilde{\Ve}_\VN\|^2_{\B{A}}
	&= \ol{A}_{\eff,\VN} - \Aeff,
	&
	\|\nabla u - \nabla u_{\VN,p}\|^2_{\B{A}}
	&= \Aeffem - \Aeff.
	\end{align*}
\end{corollary}
%..................
\begin{proof}
	This is a consequence of previous lemma for $\mathcal{X}=\cE$.
	Then $\Ve=\nabla u$ is the minimiser of homogenisation problem \eqref{eq:homog_problem} and $\Ve_\VN, \widetilde{\Ve}_\VN, \nabla u_{\VN,p}\in\mathcal{X}$ are conforming fields obtained from FFTH-Ga, FFTH-GaNi, and FEM;
	their evaluation according to \eqref{eq:FEM_minimisation}, \eqref{eq:GA_minimisation}, and \eqref{eq:gani_bound} provides the corresponding homogenised properties.
\end{proof}
%-----------------------------
\begin{corollary}
	\label{lem:algebraic_error}
	Let $\MBe\in\xEN$ and $\MB{u}\in\sR^n$ be the solution of the algebraic systems \eqref{eq:linsys_FFTHga} and \eqref{eq:linsys_FEM} obtained from FFTH-Ga and FEM.
	And let $\MBe\iter{i}\in\xEN$ and $\MB{u}\iter{i}\in\sR^n$ be any vectors --- possibly approximate solutions obtained from the iterative solver (CG) after $i$ iterations.
	Then the squared energetic norm of the algebraic error is equal to the following error in homogenised properties
	\begin{align*}
	%\label{eq:algebraic_error}
	\|\MBe-\MBe\iter{i}\|_{\MBA} &= A_{\eff,\VN,(i)} - \Aeffth,
	&
	\|\MB{u}-\MB{u}\iter{i}\|_{\MBA} &= A_{\eff,\VN,p,(i)} - \Aeffem,
	\end{align*}
	where $A_{\eff,\VN,(i)}$ and $A_{\eff,\VN,p,(i)}$ are the homogenised properties obtained by evaluation of the homogenised problem with fields determined by $\MBe\iter{i}\in\xEN$ and $\MB{u}\iter{i}\in\sR^n$.
\end{corollary}
%.........................
\begin{proof}Because of analogy, the proof is discussed only for FFTH-Ga. Let $\Ve$ and $\Ve\iter{i}$ be trigonometric polynomials determined by vectors $\MBe$ and $\MBe\iter{i}$ with respect to the trigonometric basis.
	Then, the discrete energetic norm is equal to the continuous one
	\begin{align*}
	\|\MBe-\MBe\iter{i}\|_{\MBA} = \bilf{\Ve-\Ve\iter{i}}{\Ve-\Ve\iter{i}}.
	\end{align*}
	Because of that, we can incorporate Lemma~\ref{lem:ener_norm} with $\mathcal{X}=\cEN$ and corresponding homogenised properties
	\begin{align*}
	A_{\eff,\VN,(i)} &= \bilf{\VE+\Ve\iter{i}}{\VE+\Ve\iter{i}},
	&
	A_{\eff,\VN} &= \bilf{\VE+\Ve}{\VE+\Ve}.
	\end{align*}

\end{proof}

\section{Approximate homogenised properties}
\label{sec:approximate-homogenised-properties}

Two possible ways for obtaining approximate, more reliable, homogenised properties are presented in the following sections.

\subsection{Mean of primal-dual values}\label{sec:mean-of-primal-dual-values}
According to \cite{Willis1989,Dvorak1993master}, more reliable approximate homogenised properties can be obtained from a mean values between the primal and the dual formulation
\begin{align*}
\B{A}_\eff \approx \frac{1}{2}\bigr(\B{A}_{\eff,\VN}-(\B{B}_{\eff,\VN})^{-1}\bigl),
\end{align*}
where $\mB_{\eff,\VN}$ is an approximation of the homogenised properties from the dual formulation, which is a lower bound on the homogenised properties $(\mB_{\eff,\VN})^{-1}\leq\Aeff$, see \cite{VoZeMa2014GarBounds} for details in the FFT-based setting.

\subsection{Non-linear least-squares estimation for different discretisation grids}\label{sec:non-linear-least-squares-estimation-for-different-discretisation-grids}

Another possibility for obtaining more accurate homogenised properties lies in estimation from the convergence rate using the solutions on several discretisation grids. From those convergence rate formulas discussed in Section~\ref{sec:discretisation-error}, we can derive the formula for the approximate homogenised properties $\tAeff$ using
\begin{align*}
\mathcal{A}_{\eff,N}(N,\tAeff, s, C) = C\cdot N^{-2s} + \tAeff,
\end{align*}
from which the convergence rate and homogenised properties can be estimated for a set of calculated data points $\bigl\{(N(i),\Aeffth(i))\bigr\}_{i=1}^m$ for different discretisation points.

Therefore, we define the non-linear least squares problem
\begin{align*}
(\tAeff, s, C) = \argmin_{\tAeff, s, C}\sum_{i=1}^m [\Aeffth(i)-\mathcal{A}_{\eff,N}(N(i),\tAeff, s, C)]^2.
\end{align*}
In our setting\footnote{This problem was solved with an open source numerical library \href{https://www.scipy.org/}{SciPy}, using
\href{http://docs.scipy.org/doc/scipy/reference/generated/scipy.optimize.curve_fit.html}{\texttt{scipy.optimize.curve\_fit}}}, it was enough to use four or five different meshes (discretisation grids).

\subsection{Approximate discretisation error}
Because the homogenised properties $\Aeff$ are unknown and can be only approximated, we estimate them with $\tAeff$, as described in \ref{sec:approximate-homogenised-properties}. It leads to \emph{approximate discretisation error in energetic norm}
\begin{subequations}
	\label{eq:ener_norm_app}
	\begin{align*}
	\|\Ve-\Ve_\VN\|^2_{\B{A}} &\approx \Aeffth - \tAeff,
	\\
	\| \Ve-\nabla u_{\VN,p} \|^2_{\B{A}} &\approx \Aeffem - \tAeff,
	\end{align*}
\end{subequations}
which are used for numerical comparison.
Particularly, the following values in Table~\ref{tab:Aeff} have been predicted using the homogenised properties on several grids, presented in \ref{sec:non-linear-least-squares-estimation-for-different-discretisation-grids}. Those values have been also controlled by the homogenised values obtained as averages of upper-lower bounds, presented in \ref{sec:mean-of-primal-dual-values}.

\begin{table}[h]
	\centering
	\begin{tabular}{c|c|c|c}
		dimension & problem & $\rho$ & $\tAeff$ \\
		\hline 2 & \squar & 10 & 3.0416470728
		\\
		2 & \pyramid & 10 & 3.6685617065
		\\
		3 & \squar & 10 & 2.9072530862
		\\
		3 & \pyramid & 10 & 2.9870480854
		\\
		2 & \squar & 100 & 3.6931324468
		\\
		2 & \pyramid & 100 & 14.482810295
		\\
		3 & \squar & 100 & 3.6418887304
		\\
		3 & \pyramid & 100 & 9.1891217513
		\\
		\hline
	\end{tabular}
	\caption{List of approximated homogenised properties $\tAeff$}
	\label{tab:Aeff}
\end{table}

%\section*{References}

%\bibliographystyle{elsarticle-num}
%\bibliography{library}

\end{document}